\documentclass[12pt,reqno]{amsart}
\usepackage{amsfonts,amssymb,amscd,amsmath,amsrefs,latexsym,enumerate,verbatim,calc,
quotmark,enumitem}
\usepackage[top=3cm, bottom=2.5cm,left=2.5cm, right=2.5cm]{geometry}
\usepackage[linkcolor = red, colorlinks = true]{hyperref}
\allowdisplaybreaks
\newcommand{\numberthis}[1]{\addtocounter{equation}{1}\tag{\theequation}\label{#1}}

\def\Aut{\operatorname{Aut}}

\def\hom{\operatorname{Hom}}

\def\O{{\operatorname{O}}}          
\def\EO{{\operatorname{EO}}}          

\newtheorem{theorem}{Theorem}[section]
\newtheorem{lem}[theorem]{Lemma}
\newtheorem{cor}[theorem]{Corollary}

\theoremstyle{definition}
\newtheorem{rem}[theorem]{Remark}

\newtheorem{defn}[theorem]{Definition}
\newtheorem{notn}[theorem]{Notation}

\title[Yoga of Commutators]{{{ Yoga of Commutators in Roy's Elementary Orthogonal Group}}}
\author{A. A. Ambily}
\address{\newline Statistics and Mathematics Unit
\newline Indian Statistical Institute
\newline Bangalore 560 059
\newline India.
\newline {\emph{e-mail: } \tt ambily@isibang.ac.in}}


\begin{document}
\begin{abstract} In this article, we give explicit proofs of certain commutator relations among the elementary generators of the elementary orthogonal group $\EO_A(Q\!\perp\! H(P))$, where $A$ is a commutative ring, $Q$ is a non-singular quadratic $A$-space and $H(P)$ is the hyperbolic space of a finitely generated projective module $P$ with the natural quadratic form. Using these relations, we established a local-global principle of D. Quillen for the Dickson--Siegel--Eichler--Roy (DSER) elementary orthogonal transformations in~\cite{aarr}. In \cite{aa1}, by using these commutator relations, we prove the normality of this elementary group in the orthogonal group under some conditions  on the hyperbolic rank. Also, these relations are used to obtain further information about this orthogonal group and in comparing it with similar groups such as Hermitian groups and odd unitary groups.
\end{abstract}
\maketitle
\section{Introduction}

Commutator relations involving elementary matrices play a key role
in answering questions in the K-theory of rings. Steinberg's
celebrated commutator formulae were generalized to the setting of
Chevalley-Demazure group schemes over commutative rings by Michael
Stein. The commutator formulae were pivotal in obtaining
local-global principles for various groups. Localisation is one of
the most powerful tools in the study of structure of quadratic
modules and  more generally, of algebraic groups over rings. It
helps to reduce many important problems over arbitrary commutative
rings to similar problems for semi-local rings. Localisation comes
in a number of versions such as {\bf localisation and patching},
proposed by D. Quillen in \cite{MR0427303} and A.~A. Suslin in \cite{MR0472792}, 
and {\bf localisation-completion}, proposed
by A. Bak (see \cite{MR1115826}). Both of these methods rely on the
{\it yoga of commutators}. This term was coined by R. Hazrat, A. Stepanov,
N.~A. Vavilov, Z. Zhang (see \cite{MR2822507}) and stands for a large body of
common calculations, known as {\bf conjugation calculus} and as {\bf
commutator calculus}. Their main objective is to obtain explicit
estimates of the modulus of continuity in $s$-adic topology for
conjugation by a specific matrix, and to calculate mutual commutator
subgroups, nilpotent filtration etc.

In this article, we consider a group of transformations defined and
studied by A. Roy in his Ph.\,D. thesis (see \cite{MR0231844})
generalizing the classical Eichler-Siegel transformations to
commutative rings. These elementary transformations are defined for
quadratic spaces $Q\!\perp\! H(P)$ with a hyperbolic summand over a
commutative ring and we call this elementary orthogonal group as 
Dickson--Siegel--Eichler--Roy or DSER group.
We establish several commutator relations among
Roy's elementary transformations. We have used these relations to
establish a local-global principle for Roy's group of orthogonal
transformations over a polynomial extension(see \cite{aarr}). In
addition, we have also used these commutator formulae to show that
Roy's elementary orthogonal group is normalised by the orthogonal group of smaller size and under certain stable range conditions the elementary orthogonal group is normal in the full
orthogonal group in \cite{aa1}. Further, in the same article, we proved a stability result for $K_1$ of the orthogonal group using the commutator formulae proved in this article. Thus, the commutator
formulae developed here have proved to be quite useful.

{\it Even though we needed only commutator relations in the above
applications most of the time, the relations themselves could be
found only after computing the commutators explicitly. Thus, we are
forced to work out the rather involved expressions for commutators
appearing in this article. Our commutator formulae are done by hand
although the shape emerged in a few small-dimensional cases using
the computer algebra system GAP} (see \cite{gap}).

\section{Preliminaries}\label{prelim}
Let $A$ be a commutative ring in which $2$ is invertible. A {\em quadratic $A$-module} is a pair $(M,q)$, where $M$ is an $A$-module and $q$ is a quadratic form on $M$. A {\em quadratic space} over $A$
is a pair $(M,q)$, where $M$ is a finitely generated projective $A$-module and $q:M\longrightarrow A$ is a non-singular quadratic form. Let $M^*$ denote the dual of the module $M$. Let $B_q$ be the symmetric bilinear form associated to $q$ on $M$, which is given by $B_q(x,y) = q(x+y)-q(x)-q(y)$ and $d_{B_q} : M \rightarrow M^*$ be the induced isomorphism given by $d_{B_q}(x)(y) = B_q(x,y)$, where $x,y\in M$. Given two quadratic $A$-modules $(M_1,q_1)$ and $(M_2,q_2)$, their orthogonal sum $(M,q)$ is defined by taking
$M=M_1\oplus M_2$ and $q((x_1,x_2))=q_1(x_1)+q_2(x_2)$ for $x_1 \in M_1, x_2 \in M_2$. Denote $(M,q)$ by $(M_1,q_1)\!\perp\!(M_2,q_2)$ and $q$ by $q_1 \!\perp\! q_2$.

Let $P$ be a finitely generated projective $A$-module. The module $P \oplus P^*$ has a natural quadratic form given by $p((x,f)) = f(x)$ for $x\in P$, $f\in P^*$. The corresponding bilinear form $B_p$ is given
by $B_p((x_1,f_1),(x_2,f_2)) = f_1(x_2)+f_2(x_1)$ for $x_1,x_2 \in P$ and $f_1,f_2 \in P^*$. The quadratic space $(P \oplus P^*, p)$, denoted by $H(P)$, is called the {\em hyperbolic space} of $P$. A quadratic space
$M$ is said to be hyperbolic if it is isometric to $H(P)$ for some $P$. The quadratic space $H(A)$, denoted by $h$, is called a \emph{hyperbolic plane}. The orthogonal sum $h\!\perp\! h \!\perp\! \cdots \!\perp\! h$
of $n$ hyperbolic planes is denoted by $h^n$.

Let $Q$ be a quadratic $A$-space and $P$ be a finitely generated projective $A$-module. Now let $M = Q \!\perp\! H(P)$. This is a quadratic space with the quadratic form $q\!\perp\! p$. The associated bilinear form on $M$, denoted by $\langle \cdot , \cdot\rangle$, is given by
\begin{equation*}
\langle (a,x),(b,y) \rangle = B_q((a,b)) + B_p((x,y))\textnormal{ for all } a,b \in Q \textnormal{ and } \,x,y \in H(P),
\end{equation*} where $B_q$ and $B_p$ are the bilinear forms on $Q$ and $P$.  Let $M = M(B,q)$ be a quadratic module over $A$ with quadratic form $q$ and associated symmetric bilinear form $B$. Then the orthogonal group of $M$ is defined as follows:
\begin{equation}\label{oam}
\O_A(M) = \{\sigma \in \Aut(M)\mid q(\sigma(x))=q(x) \textnormal{ for all } x \in M\},
\end{equation} where $\Aut(M)$ be the group of all $A$-linear automorphisms of $M$.

For any $A$-linear map $\alpha : Q \rightarrow P$($\beta :
Q\rightarrow P^*$), the dual map $\alpha^t : P^* \rightarrow Q^*$
($\beta^t : P^{**}\simeq P \rightarrow Q^*$) is defined as
$\alpha^t(\varphi) = \varphi \circ \alpha$ ($\beta^t(\varphi^*) =
\varphi^* \circ \beta$) for $\varphi \in P^*$ ($\varphi^* \in P^{**}$). Recall from
\cite{MR0231844}, the $A$-linear map $\alpha^* : P^*\rightarrow Q$
($\beta^* : P\rightarrow Q$) is defined by $\alpha^* = d_{B_q}^{-1}\circ
\alpha^t$ ($\beta^* = d_{B_q}^{-1}\circ \beta^t \circ \varepsilon$,
where $\varepsilon$ is the natural isomorphism $P\rightarrow P^{**}$)
and is characterized by the relation
\[(f\circ\alpha)(z) = B_q\left(\alpha^*(f),z \right)  \textnormal{ for } \,f\in
P^*, z\in Q.\]

In \cite{MR0231844}, A. Roy defined the ``elementary'' transformations
$E_{\alpha}, E_{\beta}^*$ of $Q\!\perp\! H(P)$ given by
\[\begin{array}{lllll}
\medskip
E_{\alpha}(z) &= z+\alpha(z) & & E_{\beta}^*(z) &=z+\beta(z)\\
\medskip
E_{\alpha}(x) &= x & & E_{\beta}^*(x) &= -\beta^*(x)+ x-\frac{1} {2}\beta\beta^*(x)\\
\medskip
E_{\alpha}(f) &= -\alpha^*(f)-\frac{1}{2}\alpha\alpha^*(f)+f & & E_{\beta}^*(f) &= f
\end{array}\]
for $z \in Q, x \in P$ and $f\in P^*$.  Observe that
these transformations are orthogonal with respect to the above
quadratic form $q \!\perp\! p.$

We found it difficult to give a meaningful set of commutator relations
for the set of generators
$\{E_{\alpha}, E^*_{\beta}~|~\alpha \in \hom_A(Q, P), \beta \in \hom_A(Q,P^*)\}$.

Let $Q$ and $P$ be free $A$-modules. In this case, we could conceive of a
natural set of generators, for which we could develop the commutator
machinery. These generators will be denoted by $E_{\alpha_{ij}}$,
$E^*_{\beta_{ij}}$ below. We proceed to define these now.

Let $P$ and $Q$ be free modules of rank $m$ and $n$
respectively, then we can identify $P$, $P^*$ and $Q$ with $A^m$, $A^m$
and $A^n$ respectively. Let $\{z_i : 1\leq i\leq n \}$ be a basis for
$Q$, $\{g_i : 1\leq i\leq n \}$ be a basis for $Q^*$, $\{x_i : 1\leq
i\leq m \}$ be a basis for $P$ and $\{f_i : 1\leq i\leq m \}$ be a basis
for $P^*$.

Let $p_i: A^n \longrightarrow A$ be the projection onto the $i^{th}$
component and $\eta_i: A \longrightarrow A^n$ be the inclusion into
the $i^{th}$ component.  Let $\alpha \in \hom(Q,P)$. Let $\alpha_{i},
\alpha_{ij} \in \hom(Q,P)$ be the maps given by
$$\alpha_{i} = \eta_i\circ p_i\circ\alpha \quad\textnormal{and}\quad \alpha_{ij} =
\eta_i\circ p_i\circ\alpha\circ \eta_j\circ p_j$$
for $1 \leq i \leq m$ and $1\leq j \leq n$.
Clearly $\alpha = {\Sigma_{i=1}^m\alpha_{i}}
= {\Sigma_{i=1}^m\Sigma_{j=1}^n\alpha_{ij}}.$ Then
$\alpha_{i}^*,\alpha_{ij}^* \in \hom(P^*,Q)$ be the maps given by
$$\alpha_{i}^* = {(\alpha^*)}_i = \alpha^*\circ \eta_i\circ p_i \quad \textnormal{and}\quad \alpha_{ij}^* = (\alpha^*)_{ij}= \eta_j\circ p_j\circ\alpha^*\circ \eta_i\circ p_i.$$  Then $\alpha^* = {\Sigma_{i=1}^m\alpha_{i}^*} = {\Sigma_{i=1}^m\Sigma_{j=1}^n\alpha_{ij}^*}$. One can also see that this definition of $\alpha_{i}^*,\alpha_{ij}^*$ coincides with the one obtained by applying $\alpha^* = {d_{B_q}}^{-1}\circ \alpha^t \in \hom(P^*,Q)$ to $\alpha_{i}$ and $\alpha_{ij}$.

\vspace{2mm} Let $z = {\Sigma_{j=1}^n d_jz_j}\in Q$ for some $d_j \in
A$.  Then $\alpha$ is defined by  $\alpha(z_j) = x^{(j)} =
{\Sigma_{i=1}^m b_{ij}x_i} $  for some $b_{ij} \in A$ and  $\alpha(z)
= {{\Sigma_{j=1}^n\Sigma_{i=1}^m d_jb_{ij} x_i}}$, $\alpha_i(z) =
{{\Sigma_{j=1}^n d_jb_{ij} x_i}}$ and  $\alpha_{ij}(z) = d_jb_{ij}
x_i$.  Let $\alpha^*(f_i) = w_i$ for some $w_i \in Q$. If $f = {\Sigma_{i=1}^m c_if_i}$ for some $c_i \in A$, then $c_i = \langle
f,x_i \rangle$ and so $\alpha^*(f) = {\Sigma_{i=1}^m \langle f,x_i
\rangle w_i}$. If $w_i = {{\Sigma_{j=1}^n y_jz_j}}$ for some $y_j \in
A$, then $w_{ij} = y_jz_j\in Q$.

For $1 \leq i \leq m$ and $1 \leq j \leq n$, the maps $\alpha_i^*$ and
$\alpha_{ij}^*$'s are given by
\[ \begin{array}{llll}
 \alpha_i^*(f_j) = \begin{cases}
              w_i &\textnormal{if}\quad  j=i,\\
              0 &\textnormal{if}\quad    j \neq i.
           \end{cases} & & &
\alpha_{ij}^*(f_k) = \begin{cases}
             w_{ij}&\textnormal{if}\quad k=i,\\
             0 &\textnormal{if}\quad k \neq i.
              \end{cases}
      \end{array} \]
Let $\beta \in \hom(Q,P^*)$. Set $\beta^*(x_i) = v_i$ for
some $v_i \in Q$, let $v_{ij}$ denotes the element $\eta_j\circ
p_j(v_i) $. Now defining the maps $\beta_i, \beta_{ij}, \beta_i^*$,
$\beta_{ij}^*$ similarly and extending these to the whole of $Q\oplus
P\oplus P^{\ast}$, we will get the maps as follows:  For $z\in Q$,
$x\in P$, $f\in P^*$; $1\leq i\leq m$ and $1\leq j\leq n$;
\[\begin{array}{llllll}
\medskip
\alpha_{ij}(z,x,f) &= ( 0,\langle w_{ij},z \rangle x_i,0 ),  & &\beta_{ij}(z,x,f) &=
\left(0,0,\langle v_{ij},z \rangle f_i\right),\\

\medskip
\alpha_i(z,x,f) &= \left( 0,\langle w_i,z \rangle x_i,0 \right), &
&\beta_i(z,x,f) &= \left(0,0,\langle v_i,z \rangle f_i\right),\\
\medskip
\alpha(z,x,f) &= \left( 0,{\Sigma_{i=1}^m \langle w_i,z \rangle
x_i},0\right),& &\beta(z,x,f) &= \left(0,0,{\Sigma_{i=1}^m \langle
v_i,z \rangle f_i}\right),\\
\medskip
\alpha_{ij}^*(z,x,f) &= \left(\langle f,x_i \rangle w_{ij},0,0
\right),& &\beta^*_{ij}(z,x,f) &= \left(\langle x,f_i \rangle
v_{ij},0,0 \right),\\
\medskip
\alpha_i^*(z,x,f) &= \left(\langle f,x_i \rangle w_i,0,0 \right),&
&\beta^*_i(z,x,f) &= \left(\langle x,f_i \rangle v_i,0,0 \right),\\
\medskip
\alpha^*(z,x,f) &= \left({\Sigma_{i=1}^m \langle f,x_i \rangle
w_i},0,0 \right), & &\beta^*(z,x,f) &= \left({\Sigma_{i=1}^m \langle
x,f_i \rangle v_i},0,0 \right).
\end{array}\]
Also, $q(w_{ij}) = \frac{1}{2}\langle w_{ij},w_{ij}
\rangle$ and $q(v_{ij}) = \frac{1}{2}\langle v_{ij},v_{ij} \rangle$.

\vspace{3mm}
The orthogonal transformation $E_{{\alpha}_{ij}}$ on $Q\!\perp\! H(P)$ for $\alpha \in \hom(Q,P)$ is
given by
\begin{align*} E_{\alpha_{ij}}(z,x,f) =&
\Bigl(I-\alpha^*_{ij}+\alpha_{ij}-\frac{1}{2}\alpha_{ij}\alpha^*_{ij}\Bigr)(z,x,f)\\
=& \Bigl(z-\langle f,x_i \rangle w_{ij},\;x+\langle w_{ij},z\rangle
x_i -\langle f,x_i \rangle q(w_{ij})x_i,\;f\Bigr).
\end{align*}
The orthogonal transformation $E_{{\beta}_{ij}}^*$ of $Q\!\perp\! H(P)$ for $\beta \in \hom(Q,P^*)$ is
given by
\begin{align*} E_{\beta_{ij}}^*(z,x,f) = &
\Bigl(I-\beta_{ij}^*+\beta_{ij}-\frac{1}{2}\beta_{ij}\beta_{ij}^*\Bigr)(z,x,f)\\
=& \Bigl(z-\langle f_i,x \rangle v_{ij},\;x,\;f+\langle
v_{ij},z\rangle f_i -\langle x,f_i \rangle q(v_{ij})f_i\;\Bigr).
\end{align*}

\begin{notn}
 Let $G$ be a group and $a,b \in G$. Then $[a,b]$ denotes the commutator $aba^{-1}b^{-1}$.
\end{notn}

\section{Commutators of Elementary Transformations}
In this section, we establish various commutator relations among the elementary generators of Roy's elementary orthogonal group. We will carry out the computations in two different ways - one is by
choosing bases (which we call the method {\it using coordinates}), and the other by just using the formal definition without choosing bases (which we call the {\it coordinate-free method}). We need commutator relations of length up to 16. By the \textquoteleft length\textquoteright of a commutator, we mean the number of words in the commutator expression.

The following is a coordinate-free definition of the elementary generators.
\begin{defn}
For $\theta \in \hom_A(Q ,P)$ or $\hom_A(Q, P^*)$, define $\theta^*$ as $d_{B_q}^{-1}\circ \theta^t$ or $d_{B_q}^{-1}\circ \theta^t \circ \varepsilon$, where $\varepsilon$ is the natural isomorphism $P\rightarrow P^{**}$ according to whether $\theta \in \hom_A(Q,P)$ or $\hom_A(Q,P^*)$ respectively.
Then the elementary transformations $E_{\theta}$ and $E_{\theta}^{-1}$ are given by 
\begin{align*}
 E_{\theta} &= I + \theta - \theta^* - \frac{1}{2} \theta \theta^*,\\
 E_{\theta}^{-1}  &= I - \theta + \theta^* - \frac{1}{2} \theta \theta^* = E_{(-\theta)}.
\end{align*}
\end{defn}

We now give the definition of the elementary generators using coordinates.
\begin{defn}
Let $\alpha,\delta \in \hom_A(Q,P)$; $\beta,\gamma \in \hom_A(Q,P^*)$
and $w_i,t_i,v_i,c_i  \in Q$ for $1 \leq i \leq m$. Then, choosing
bases $\{x_i\}_{i=1}^m, \{f_i\}_{i=1}^m,\{z_i\}_{i=1}^m$ respectively for $P,P^*,Q$,
one can define the following elements in $\hom_A(Q \perp H(P))$.
      $$\begin{array}{llllllll}
	\medskip
        \alpha_{ij}\left(z,x,f\right) &= & (0,\langle w_{ij},z\rangle x_i,0), & & \alpha^{\ast}_{ij}\left(z,x,f\right) &= & (\langle f,x_i \rangle w_{ij},0,0),\\
        \medskip
        \delta_{kl}\left(z,x,f\right) &= & (0,\langle t_{kl},z\rangle x_k,0), & & \delta^{\ast}_{kl}\left(z,x,f\right) &= & (\langle f,x_k \rangle t_{kl},0,0),\\
        \medskip
        \beta_{ij}(z,x,f) &= &(0,0,\langle v_{ij},z\rangle f_i), & &
        \beta^{\ast}_{ij}(z,x,f) &= &(\langle x,f_i \rangle v_{ij},0,0),\\
        \medskip
        \gamma_{kl}(z,x,f) &= &(0,0,\langle c_{kl},z\rangle f_k), & &
        \gamma^{\ast}_{kl}(z,x,f) &= &(\langle x,f_k \rangle c_{kl},0,0).
        \end{array}$$
Here $w_{ij}, v_{ij}$ denote the elements $\eta_j\circ p_j(w_i) ,
\eta_j\circ p_j(v_i) $ respectively and $c_{kl}, t_{kl}$ denote the
elements $\eta_l\circ p_l(c_k),\eta_l\circ p_l(t_k)$, where $p_j$ is
the $j^{th}$ projection as defined in Section~\ref{prelim}.

Now, for $1 \leq i,k \leq m$ and $1 \leq j,l \leq n$, the
corresponding orthogonal transformations $E_{\alpha_{ij}},
E_{\delta_{kl}}, E_{\beta_{ij}}^*, E_{\gamma_{kl}}^*$ and their
inverses have the following form.
\begin{align*}
   E_{\alpha_{ij}}\left(z,x,f\right)
                              =\;& \Bigl(z-\langle f,x_i \rangle w_{ij},\;x+\langle w_{ij},z\rangle x_i - \langle f,x_i   \rangle q(w_{ij})x_i,\;f\Bigr),\\
   E_{\delta_{kl}}\left(z,x,f\right)
                              =\;& \Bigl(z-\langle f,x_k \rangle t_{kl},\;x+\langle t_{kl},z\rangle x_k  - \langle f,x_k \rangle q(t_{kl})x_k,\;f\Bigr),\\
   E_{\beta_{ij}}^*\left(z,x,f\right)
                              =\;& \Bigl(z-\langle f_i,x \rangle v_{ij},\;x,\;f+\langle v_{ij},z\rangle f_i  - \langle x,f_i \rangle q(v_{ij})f_i\Bigr),\\
   E_{\gamma_{kl}}^*\left(z,x,f\right)
                              =\;& \Bigl(z-\langle f_k,x \rangle c_{kl},\;x,f+\langle c_{kl},z\rangle f_k  - \langle x,f_k \rangle q(c_{kl})f_k\Bigr),\\
   E_{\alpha_{ij}}^{-1}\left(z,x,f\right)
                              =\;& \Bigl(z+\langle f,x_i \rangle w_{ij} , x-\langle w_{ij},z\rangle x_i
                        -\langle f,x_i \rangle q(w_{ij})x_i , f\Bigr),\\
   E_{\delta_{kl}}^{-1}\left(z,x,f\right)
                              =\;& \Bigl(\;z+\langle f,x_k \rangle t_{kl} , x-\langle t_{kl},z\rangle x_k
                        -\langle f,x_k \rangle q(t_{kl})x_k , f\Bigr),\\
   E_{\beta_{ij}}^{*^{-1}}\left(z,x,f\right)
                              =\;& \Bigl(z+\langle f_i,x \rangle v_{ij},\;x,\;f-\langle v_{ij},z\rangle f_i  - \langle x,f_i \rangle q(v_{ij})f_i\Bigr),\\
   E_{\gamma_{kl}}^{*^{-1}}\left(z,x,f\right)
                              =\;& \Bigl(z+\langle f_k,x \rangle c_{kl},\;x,f-\langle c_{kl},z\rangle f_k  - \langle x,f_k \rangle q(c_{kl})f_k\Bigr).
\end{align*}
\end{defn}
\medskip
The first (and the simplest) set of commutators which we compute is
between elementary generators corresponding to two elements of
$\hom_A(Q,P)$; this is given in the following lemma.

\begin{lem}\label{l01}
      Let $\alpha, \delta \in \hom_A(Q,P)$. Then, for $i,j,k,l$ with $1 \leq i,k \leq m$ and ${1 \leq j,l \leq n}$, the commutator of the type
      $\Bigl[ E_{\alpha_{ij}}, E_{\delta_{kl}}  \Bigr]$ is given by
        \begin{align*}
          \Bigl[ E_{\alpha_{ij}}, E_{\delta_{kl}}  \Bigr](z,x,f)
                  =& \Bigl( I + \delta_{kl}\alpha^{\ast}_{ij} -
                      \alpha_{ij} \delta^{\ast}_{kl} \Bigr)\left( z,x,f \right)\\
              =& \Bigl(z,\;x + \langle f,x_i \rangle \langle t_{kl},w_{ij} \rangle x_k -
              \langle f,x_k \rangle \langle w_{ij}, t_{kl}\rangle x_i,\;f\Bigr).
    \end{align*}
In particular, if $i=k$, then $\Bigl[ E_{\alpha_{ij}}, E_{\delta_{kl}}\Bigr] = I.$
\end{lem}
\begin{proof}
For $\alpha, \delta \in \hom_A(Q,P)$ and for any $i,j,k,l$ with $1 \leq i,k \leq m$ and
$1 \leq j,l \leq n$, using the coordinate-free definition of the elementary generators, we have
     \begin{align*}
      \Bigl[ E_{\alpha_{ij}},& E_{\delta_{kl}} \Bigr](z,x,f)\\
             &=E_{\alpha_{ij}} E_{\delta_{kl}}E^{-1}_{\alpha_{ij}} E^{-1}_{\delta_{kl}}(z,x,f)\\
             &=E_{\alpha_{ij}} E_{\delta_{kl}}E^{-1}_{\alpha_{ij}}\Bigl(\Bigl( I - \delta_{kl} + \delta_{kl}^* - \frac{1}{2} \delta_{kl}\delta_{kl}^* \Bigr)(z,x,f)\Bigr)\\
             &= E_{\alpha_{ij}} E_{\delta_{kl}}
             \Bigl(\Bigl( I - \delta_{kl} + \delta_{kl}^* - \frac{1}{2} \delta_{kl}\delta_{kl}^* - \alpha_{ij} + \alpha_{ij}^* - \frac{1}{2} \alpha_{ij}\alpha_{ij}^* - \alpha_{ij} \delta_{kl}^*\Bigr)(z,x,f)\Bigr)\\
             &= E_{\alpha_{ij}}\Bigl(\Bigl( I - \alpha_{ij} + \alpha_{ij}^* - \frac{1}{2} \alpha_{ij} \alpha_{ij}^* - \alpha_{ij} \delta_{kl}^* + \delta_{kl} \alpha_{ij}^*\Bigr)(z,x,f) \Bigr)\\
             &= \Bigl( I - \alpha_{ij} \delta_{kl}^* + \delta_{kl} \alpha_{ij}^* \Bigr)(z,x,f).
       \end{align*}
     Using coordinates, we may compute the above commutator as
          \begin{align*}
      \Bigl[ E_{\alpha_{ij}},& E_{\delta_{kl}} \Bigr](z,x,f)\\
             &= E_{\alpha_{ij}} E_{\delta_{kl}}E^{-1}_{\alpha_{ij}}
            \Bigl(z + \langle f,x_k \rangle t_{kl} , \; x-\langle t_{kl},z\rangle x_k
            -\langle f,x_k \rangle q(t_{kl})x_k ,\; f \Bigr)\\
            &= E_{\alpha_{ij}} E_{\delta_{kl}} \Bigl(z+ \langle f , x_i\rangle w_{ij}
              +\langle f,x_k \rangle t_{kl},\;
              x-\Bigl\{
              \langle w_{ij}, z \rangle
              + \langle f , x_i \rangle q\left(w_{ij}\right) \Bigr.\Bigr.\\
            &\hspace{1cm}\Bigl.\Bigl. +\langle f,x_k \rangle \langle w_{ij},t_{kl} \rangle \Bigr\} x_i
              -\Bigl\{\langle t_{kl},z\rangle + \langle f,x_k \rangle q(t_{kl})\Bigr\}x_k,\; f \Bigr)\\
                &= E_{\alpha_{ij}} \Bigl(\,z+\langle f,x_i \rangle w_{ij},\;x-
                \Bigl\{\langle w_{ij},z\rangle +  q(w_{ij})\langle f , x_i \rangle  +\langle f,x_k \rangle
                \langle w_{ij},t_{kl} \rangle \Bigr\}x_i \Bigr.\Bigr.\\
                &\hspace{1cm}\Bigl.\Bigl.+ \langle f,x_i \rangle \langle t_{kl},w_{ij} \rangle x_k , f \Bigr)\\
             &= \Bigl(\, z,\,x+\langle f,x_i \rangle \langle t_{kl},w_{ij} \rangle x_k - \langle f,x_k \rangle \langle w_{ij}, t_{kl}\rangle x_i,\, f \Bigr).
     \end{align*}
       If $i = k$, then we have
     \[
      \delta_{kl} \alpha^{\ast}_{ij}(z,x,f)   = \Bigl(0,\langle f,x_i \rangle \langle t_{il},w_{ij} \rangle x_i,0 \Bigr) = \alpha_{ij} \delta^{\ast}_{kl}(z,x,f).
    \]
    Hence $\Bigl[ E_{\alpha_{ij}}, E_{\delta_{il}} \Bigr] = I.$
\end{proof}

As a consequence of this lemma, we have the following commutator relations.
\begin{cor}\label{c01}
        For any $i,j,k,l$ with $1 \leq i,k \leq m$, $1 \leq j,l \leq n$ and for $a,b,c,d \in A$ with  $ab=cd$, the following equation holds.
          \[
            \Bigl[E_{a\alpha_{ij}}, E_{b\delta_{kl}}\Bigr]
              = \Bigl[E_{c\alpha_{ij}}, E_{d\delta_{kl}}\Bigr].
           \]
\end{cor}
\begin{proof}
 For $\alpha,\delta \in \hom_A(Q,P)$ and for any $i,j,k,l$ with $1 \leq i,k \leq m$, $1 \leq j,l \leq n$ and $a,b,c,d \in A$ with $ab = cd$, we have
          \begin{align*}
            \Bigl[E_{a\alpha_{ij}}, E_{b\delta_{kl}}\Bigr]
             =\;&I - ab \alpha_{ij} \delta_{kl}^* + ab\delta_{kl} \alpha_{ij}^*  \quad\quad\textnormal{ (by Lemma~\ref{l01})}\\
            =\;& I - cd\alpha_{ij} \delta_{kl}^* + cd \delta_{kl} \alpha_{ij}^* = \Bigl[E_{c\alpha_{ij}}, E_{d\delta_{kl}}\Bigr].\qedhere
         \end{align*}
\end{proof}
{\it Since we will be using similar calculations to find the commutators 
in the rest of the article, we will give only the final expression for 
the commutators.}

We now compute the `mixed commutator' of elementary generators
corresponding to elements of $\hom_A(Q,P)$ and $\hom_A(Q,P^*)$.
The expression for the commutator as given in the proof of the 
lemma below may appear complicated and we need only its special
case $i \neq k$. This special case can be deduced after obtaining 
the general expression and specializing it.

\begin{lem}\label{l02}
      Let $\alpha \in \hom_A(Q,P)$ and $\beta \in \hom_A(Q,P^*)$. Then, for $i,j,k,l$ with $1 \leq i,k \leq m$
      and $1 \leq j,l \leq n$ with $i \neq k$,
      \begin{align*}
          \Bigl[E_{\alpha_{ij}},E_{\beta_{kl}}^{\ast}\Bigr](z,x,f)
            =\;& \Bigl(I-\alpha_{ij} \beta^{\ast}_{kl}+ \beta_{kl}\alpha^{\ast}_{ij}\Bigr)(z,x,f)\\
            =\;& \Bigl(\;z,\; x-\langle x,f_k\rangle \langle w_{ij},v_{kl}\rangle x_i,\; f+\langle f,x_i\rangle          \langle v_{kl},w_{ij}\rangle f_k\;\Bigr).
      \end{align*}
\end{lem}
\begin{proof}
 For $\alpha \in \hom_A(Q,P)$, $\beta \in \hom_A(Q,P^*)$ and for any $i,j,k,l$
 with $1 \leq i,k \leq m$ and $1 \leq j,l \leq n$ with $i \neq k$,
 we have the coordinate-free expression
 \begin{align*}
      \Bigl[ E_{\alpha_{ij}},& E_{\beta_{kl}}^* \Bigr](z,x,f)\\
            &= \Bigl(I + \beta_{kl}^* \alpha_{ij} + \frac{1}{2}\beta_{kl}^*\alpha_{ij} \alpha_{ij}^* +
      \beta_{kl}^* \alpha_{ij} \beta_{kl}^* - \frac{1}{2}\beta_{kl}^*\alpha_{ij}\alpha_{ij}^*\beta_{kl}
      - \frac{1}{4} \beta_{kl}^*\alpha_{ij}\alpha_{ij}^*\beta_{kl} \beta_{kl}^*\Bigr.\\
      &\hspace{1cm} - \alpha_{ij}^*\beta_{kl}  -  \frac{1}{2} \alpha_{ij}^* \beta_{kl} \beta_{kl}^* - \alpha_{ij}^* \beta_{kl}\alpha_{ij}^* + \alpha_{ij}^* \beta_{kl}\alpha_{ij}^*\beta_{kl} - \frac{1}{2} \alpha_{ij}^* \beta_{kl} \beta_{kl}^*\alpha_{ij} +\beta_{kl}\alpha_{ij}^*\\
      &\hspace{1cm} - \frac{1}{4}  \alpha_{ij}^* \beta_{kl} \beta_{kl}^*\alpha_{ij} \alpha_{ij}^* + \frac{1}{4}  \alpha_{ij}^* \beta_{kl} \beta_{kl}^*\alpha_{ij} \alpha_{ij}^* \beta_{kl} +\frac{1}{8} \alpha_{ij}^* \beta_{kl} \beta_{kl}^*\alpha_{ij} \alpha_{ij}^* \beta_{kl} \beta_{kl}^* + \frac{1}{2}\beta_{kl}\beta_{kl}^*\alpha_{ij}\\
      &\hspace{1cm} - \alpha_{ij} \beta_{kl}^* + \frac{1}{2} \alpha_{ij} \alpha_{ij}^*\beta_{kl}  + \frac{1}{4} \alpha_{ij} \alpha_{ij}^*\beta_{kl} \beta_{kl}^* - \alpha_{ij} \alpha_{ij}^*\beta_{kl} - \frac{1}{2} \alpha_{ij} \alpha_{ij}^*\beta_{kl}\beta_{kl}^* + \frac{1}{4}\beta_{kl}\beta_{kl}^*\alpha_{ij}\alpha_{ij}^* \\
      &\hspace{1cm} + \alpha_{ij} \beta_{kl}^*\alpha_{ij}  + \frac{1}{2} \alpha_{ij} \beta_{kl}^*\alpha_{ij}\alpha_{ij}^* + \alpha_{ij} \beta_{kl}^*\alpha_{ij} \beta_{kl}^* -\frac{1}{2} \alpha_{ij} \beta_{kl}^*\alpha_{ij}\alpha_{ij}^* \beta_{kl}- \beta_{kl}\alpha_{ij}^* \beta_{kl}\\
      &\hspace{1cm} -\frac{1}{4}\alpha_{ij} \beta_{kl}^*\alpha_{ij}\alpha_{ij}^* \beta_{kl}\beta_{kl}^* - \frac{1}{4}\beta_{kl}\beta_{kl}^*\alpha_{ij}\alpha_{ij}^* \beta_{kl}  - \frac{1}{8} \beta_{kl}\beta_{kl}^*\alpha_{ij}\alpha_{ij}^* \beta_{kl}\beta_{kl}^* \Bigr)(z,x,f).
     \end{align*}
Now using coordinates, we have
      \begin{align*}
          \Bigl[ E_{\alpha_{ij}},& E_{\beta_{kl}}^* \Bigr](z,x,f)\\
          &= \Bigl(z +  \Bigl\{ \langle w_{ij},z \rangle + \langle x,f_k \rangle
          \langle w_{ij},v_{kl} \rangle  +\langle f,x_i \rangle
           q(w_{ij})-\langle x,f_k \rangle \langle f_k,x_i \rangle q(v_{kl})q(w_{ij})\Bigr.\\
          &\Bigl.\hspace{1cm}-\langle v_{kl}, z\rangle \langle f_k,x_i \rangle q(w_{ij})
           \Bigr\}\langle x_i,f_k\rangle v_{kl}
          -\Bigl\{ \langle v_{kl},z \rangle + \langle x,f_k \rangle  q(v_{kl})
          + \langle x_i,f \rangle \langle v_{kl},w_{ij}\rangle  \Bigr.\\
          &\hspace{.9cm}-\langle v_{kl},z \rangle \langle v_{kl},w_{ij} \rangle \langle x_i,f_k \rangle
          + \langle x_i,f \rangle\langle x_i,f_k \rangle q(v_{kl})q(w_{ij})+ \langle w_{ij},z \rangle\langle x_i,f_k \rangle q(v_{kl})
          \\
          &\hspace{.9cm}- \langle v_{kl},z \rangle\langle x_i,f_k \rangle^2 q(v_{kl})q(w_{ij})  
         \Bigl.- \langle x,f_k \rangle\langle x_i,f_k \rangle^2 q(v_{kl})q(w_{ij})\Bigr\}\langle x_i,f_k \rangle  w_{ij},\\
          & \hspace{1cm} x + \Bigl\{\langle w_{ij},z \rangle \langle v_{kl}, w_{ij} \rangle \langle f_k,x_i \rangle -\langle x,f_k \rangle \langle w_{ij}, v_{kl} \rangle
            -\langle v_{kl},z \rangle\langle x_i,f_k \rangle q(w_{ij})\Bigr. \\
          & \hspace{1cm} -\langle x,f_k \rangle\langle x_i,f_k \rangle q(w_{ij})q(v_{kl})-\langle x,f_k \rangle \langle v_{kl}, w_{ij} \rangle \langle f_k,x_i \rangle^2q(v_{kl})q(w_{ij}) 
          \\
          & \hspace{1cm}+\langle x,f_k \rangle  \langle v_{kl}, w_{ij} \rangle^2 \rangle \langle x_i,f_k\rangle 
          - \langle w_{ij}, z\rangle\langle x_i,f_k \rangle^2 q(v_{kl})q(w_{ij})\\
          &\hspace{1cm}- \langle x_i,f \rangle \langle x_i,f_k \rangle^2 q(v_{kl})q(w_{ij})^2        
          + \langle v_{kl}, z \rangle\langle x_i,f_k \rangle^3 q(v_{kl})q(w_{ij})^2 \\
          &\hspace{1cm}\Bigl.
          + \langle x,f_k\rangle \langle x_i,f_k  \rangle^3 q(v_{kl})q(w_{ij}))^2
         \Bigr\}x_i,\;
	    f+\Bigl\{\langle x_i,f \rangle \langle x_i,f_k \rangle q(v_{kl})q(w_{ij})\Bigr.\\
          &\hspace{1cm}- \langle v_{kl}, z \rangle\langle x_i,f_k \rangle^2 q(v_{kl})q(w_{ij})
           +\langle f,x_i \rangle \langle v_{kl}, w_{ij} \rangle
          + \langle w_{ij}, z\rangle\langle x_i,f_k \rangle q(v_{kl}) \\
         &\hspace{1cm} \Bigl. \Bigl. - \langle x,f_k\rangle \langle x_i,f_k  \rangle^2
              q(v_{kl})q(w_{ij})- \langle v_{kl}, z \rangle \langle v_{kl}, w_{ij} \rangle \langle x_i,f_k \rangle \Bigr\}f_k \Bigr ).
 \end{align*}
  In the special case when $i\neq k$, using the fact that
  $\langle x_i,f_k \rangle  =0$, we obtain
  \begin{equation*}
    \Bigl[ E_{\alpha_{ij}},E_{\beta_{kl}}^{\ast} \Bigr](z,x,f)
                = \Bigl(z,\; x-\langle x,f_k \rangle  \langle w_{ij},v_{kl} \rangle  x_i,
                \; f+ \langle f,x_i \rangle  \langle v_{kl},w_{ij} \rangle  f_k\Bigr).
  \end{equation*}
  \begin{align*}
     \textnormal{Now } \alpha_{ij} \beta^{\ast}_{kl}(z,x,f) =\;& \Bigl(0,\langle x,f_k \rangle  \langle w_{ij},v_{kl} \rangle  x_i,0\Bigr), & \beta_{kl} \alpha^{\ast}_{ij}(z,x,f) =\;& \Bigl(0,0,\langle f,x_i \rangle  \langle v_{kl},w_{ij} \rangle  f_k\Bigr).
  \end{align*}
\noindent Hence if $i \neq k$, then
  \begin{align*}
  \Bigl[E_{\alpha_{ij}},E_{\beta_{kl}}^{\ast}\Bigr](z,x,f)
  &=\Bigl(\;z,\; x-\langle x,f_k \rangle  \langle w_{ij},v_{kl} \rangle  x_i,\; f+ \langle f,x_i \rangle  \langle v_{kl},w_{ij} \rangle  f_k\Bigr)\\
  &=\Bigl(I-\alpha_{ij} \beta^{\ast}_{kl}+ \beta_{kl}\alpha^{\ast}_{ij}\Bigr)(z,x,f).\qedhere
\end{align*}
 \end{proof}
The following corollary lists the resultant commutator relations from the above lemma.
\begin{cor}\label{c02}
  For any $i,j,k,l$ with $1 \leq i,k \leq m$, $1 \leq j,l \leq n$, $i \neq k$ and for $a,b,c,d \in A$ with  $ab=cd$, the following equation holds.
 \[
      \Bigl[ E_{a\alpha_{ij}},E_{b\beta_{kl}}^{\ast} \Bigr] = \Bigl[ E_{c\alpha_{ij}},E_{d\beta_{kl}}^{\ast}\Bigr].
  \]
 \end{cor}
The lemma below computes the commutator of elementary generators
corresponding to two elements of $\hom_A(Q,P^*)$.
 \begin{rem}\label{r01}
 For any $i,j,k,l$ with $1 \leq i,k \leq m$, $1 \leq j,l \leq n$
 and $i \neq k$, the commutator $\Bigl[ E_{\alpha_{ij}},E_{\beta_{kl}}^{\ast}\Bigr]^{-1}$ is given by
 \begin{align*}
 \Bigl[ E_{\alpha_{ij}},E_{\beta_{kl}}^{\ast}\Bigr]^{-1}(z,x,f)
 =\;&\Bigl(z,\; x+\langle x,f_k \rangle  \langle w_{ij},v_{kl} \rangle  x_i,\; f- \langle f,x_i \rangle  \langle v_{kl},w_{ij} \rangle  f_k\Bigr)\\
 =\;& \Bigl(I+\alpha_{ij} \beta^{\ast}_{kl}- \beta_{kl}\alpha^{\ast}_{ij}\Bigr)(z,x,f)\\
 =\;& \Bigl[ E_{\beta_{kl}}^{\ast}, E_{\alpha_{ij}}\Bigr](z,x,f).
 \end{align*}
 \end{rem}
 \begin{lem}\label{l03}
Let $\beta, \gamma \in \hom_A(Q,P^*)$. Then, for $i,j,k,l$ with $1 \leq i,k \leq m$ and $1 \leq j,l \leq n$, the commutator $[E_{\beta_{ij}}^*, E_{\gamma_{kl}}^*]$ is given by
\begin{align*}
 \Bigl[E_{\beta_{ij}}^*, E_{\gamma_{kl}}^* \Bigr](z,x,f) =\;& \Bigl( I + \gamma_{kl} \beta^{\ast}_{ij} - \beta_{ij} \gamma^{\ast}_{kl} \Bigr)(z,x,f)\\
 =\;& \Bigl(z,\; x,\; f + \langle x,f_i \rangle \langle c_{kl},v_{ij} \rangle f_k - \langle x,f_k \rangle \langle v_{ij}, c_{kl}\rangle f_i  \Bigr).
\end{align*}
In particular, if $i=k$, then $[E_{\beta_{ij}}^*, E_{\gamma_{kl}}^*] = I.$
\end{lem}
\begin{proof}
 For $\beta, \gamma \in \hom_A(Q,P^*)$ and for any $i,j,k,l$ with $1 \leq i,k \leq m$
 and $1 \leq j,l \leq n$, we have the coordinate-free expression
 \begin{align*}
      \Bigl[ E_{\beta_{ij}}^{\ast}, E_{\gamma_{kl}}^{\ast} \Bigr](z,x,f)
            =\;& E_{\beta_{ij}}^{\ast} E_{\gamma_{kl}}^{\ast} E_{\beta_{ij}}^{\ast \, -1}E_{\gamma_{kl}}^{\ast \, -1}(z,x,f)\\
            =\;& \Bigl( I - \beta_{ij}\gamma_{kl}^* + \gamma_{kl}\beta_{ij}^* \Bigr)(z,x,f).
       \end{align*}
Using coordinates, we have
\begin{align*}
\Bigl[ E_{\beta_{ij}}^*, E_{\gamma_{kl}}^*\Bigr]
=\;& \Bigl(z,\;x,\;f+\langle x,f_i \rangle \langle c_{kl},v_{ij} \rangle f_k - \langle x,f_k \rangle \langle v_{ij}, c_{kl}\rangle f_i\Bigr).
\end{align*}
If $i = k$, then
      \[
        \gamma_{kl}\beta_{ij}^*(z,x,f) = \Bigl(0,0,\langle x,f_i \rangle \langle c_{kl},v_{ij} \rangle f_k - \langle x,f_k \rangle \langle v_{ij}, c_{kl}\rangle f_i\Bigr) = \beta_{ij}\gamma_{kl}^*(z,x,f).
        \]
Hence $\Bigl[E_{\beta_{ij}}^*, E_{\gamma_{il}}^*\Bigr] = I.$
\end{proof}
Immediately, we deduce the following commutator relations.
\begin{cor}\label{c03}
    For any $i,j,k,l$ with $1 \leq i,k \leq m$, $1 \leq j,l \leq n$ and for $a,b,c,d \in A$ with  $ab=cd$, the following equation holds.
    \[\Bigl[E_{a\beta_{ij}}^*, E_{b\gamma_{kl}}^* \Bigr]= \Bigl[ E_{c\beta_{ij}}^*, E_{d\gamma_{kl}}^* \Bigr].\]
\end{cor}
 \begin{rem}
In the following sections, we will prove more complicated commutator relations of lengths $10$ and $16$; we will show how the indices may be specialized so that the commutator is non-trivial.
\end{rem}
\section{Triple Commutators}\label{triple}
In this section, we prove certain triple commutator relations among
the elementary generators of Roy's elementary orthogonal group. We
start with a commutator of length $10$ which involves a commutator of
elementary generators corresponding to two elements of $\hom_A(Q,P)$.
\begin{lem}\label{l04}
      Let $\alpha, \delta \in \hom_A(Q,P)$ and $\beta \in \hom_A(Q,P^*)$. Then, for $i,j,k,l,p,q$ with $1 \leq i,k,p \leq m$, $1 \leq j,l,q \leq n$ and $k \neq p$, the triple commutator
      $\left[ E_{\beta_{ij}}^*,\left[ E_{\alpha_{kl}},E_{\delta_{pq}} \right]\right]$ is given by
      \[
               \left[ E_{\beta_{ij}}^*,\left[ E_{\alpha_{kl}},E_{\delta_{pq}} \right]\right] =
                    \begin{cases}
                    \vspace{2mm}
                         E_{\lambda_{kj}}\left[ E_{\beta_{ij}}^*,E_\frac{\lambda_{kj}}{2} \right]&\textnormal{if}\quad i = p,\\
                         E_{\xi_{pj}}\left[ E_{\beta_{ij}}^*,E_\frac{\xi_{pj}}{2} \right]
                         &\textnormal{if}\quad i = k,\\
                         I & \textnormal{if}\quad  i \neq p \textnormal{ and } i \neq k,
                      \end{cases}
         \]
where $\lambda_{kj}\;=\;\alpha_{kl}\delta_{pq}^*\beta_{ij}$ and
 $\xi_{pj}\;=\;-\delta_{pq}\alpha_{kl}^*\beta_{ij}$.
\end{lem}
\begin{proof}
For $\alpha, \delta \in \hom_A(Q,P)$, $\beta \in \hom_A(Q,P^*)$ and for $i,j,k,l,p,q$
with $1 \leq i,k,p \leq m$, $1 \leq j,l,q \leq n$ and $k \neq p$, we have
\begin{align*}
      \left[ E_{\alpha_{kl}},E_{\delta_{pq}} \right](z,x,f)
      &= \Bigl( I + \delta_{pq} \alpha_{kl}^* - \alpha_{kl}\delta_{pq}^* \Bigr)(z,x,f)\\
      &= \Bigl(z,\,x+\langle f,x_k \rangle \langle t_{pq},w_{kl}\rangle x_p-\langle f,x_p \rangle\langle t_{pq},w_{kl}\rangle  x_k,\,f \Bigr).\\
          & \hspace{6cm}(\textnormal{by Lemma}\; \ref{l01})\\
      {\left[ E_{\alpha_{kl}},E_{\delta_{pq}} \right]}^{-1}(z,x,f)
      &= \left[ E_{\delta_{pq}},E_{\alpha_{kl}} \right](z,x,f)
      = \Bigl( I - \delta_{pq} \alpha_{kl}^* + \alpha_{kl}\delta_{pq}^* \Bigr)(z,x,f)\\
      &=\Bigl(z,\,x-\langle f,x_k \rangle \langle t_{pq},w_{kl}\rangle x_p + \langle f,x_p \rangle\langle               t_{pq},w_{kl}\rangle  x_k,\,f \Bigr).\\
      & \hspace{6cm}(\textnormal{by Lemma}\; \ref{l01})
     \end{align*}
\noindent Hence we get the coordinate-free expressions
\begin{align*}
    \left[ E_{\beta_{ij}}^*\right.,& \left[ E_{\alpha_{kl}},E_{\delta_{pq}}\right]\left.\right](z,x,f)\\
            &= E_{\beta_{ij}}^*\left[ E_{\alpha_{kl}},E_{\delta_{pq}}\right] E_{\beta_{ij}}^{*^{-1}}                    \left[ E_{\alpha_{kl}},E_{\delta_{pq}}\right]^{-1}(z,x,f)\\
            &= \Bigl( I + \beta_{ij}^* \alpha_{kl} \delta_{pq}^*  - \frac{1}{2} \alpha_{kl} \delta_{pq}^*\beta_{ij}\beta_{ij}^* \delta_{pq} \alpha_{kl}^*  + \frac{1}{2}\beta_{ij}\beta_{ij}^*          \alpha_{kl} \delta_{pq}^* +  \frac{1}{2} \alpha_{kl} \delta_{pq}^*\beta_{ij}\beta_{ij}^*\Bigr. \\
            & \hspace{1cm} 
                - \frac{1}{2} \delta_{pq} \alpha_{kl}^* \beta_{ij}\beta_{ij}^* \alpha_{kl} \delta_{pq}^*  +  \alpha_{kl} \delta_{pq}^*\beta_{ij}
                      + \frac{1}{2} \delta_{pq} \alpha_{kl}^* \beta_{ij}\beta_{ij}^* \delta_{pq}          \alpha_{kl}^* - \frac{1}{2}\beta_{ij}\beta_{ij}^*\delta_{pq} \alpha_{kl}^*  \\
             &\hspace{1cm} - \delta_{pq} \alpha_{kl}^* \beta_{ij} - \beta_{ij}^* \delta_{pq}  \alpha_{kl}^*- \frac{1}{2} \delta_{pq} \alpha_{kl}^*                \beta_{ij}\beta_{ij}^* \Bigl. +  \frac{1}{2} \alpha_{kl} \delta_{pq}^*\beta_{ij}\beta_{ij}^* \alpha_{kl} \delta_{pq}^* \Bigr)(z,x,f). \numberthis{eqnt2}
\end{align*}
\noindent On computing using coordinates, we get
\begin{align*}
 \left[ E_{\beta_{ij}}^*,\right.& \left[E_{\alpha_{kl}},E_{\delta_{pq}}\right]\left.\right](z,x,f)\\
				 &=\Bigl(z+\Bigl\{\langle f,x_p\rangle\langle x_k,f_i \rangle
				    -\langle f,x_k\rangle\langle x_p,f_i \rangle \Bigr\}\langle t_{pq},w_{kl} \rangle v_{ij},\;x- \Bigl\{\langle v_{ij},z \rangle \Bigr.\;\\
				 &\hspace{1cm} 
				  + \langle f_i,x_k \rangle\langle f,x_p \rangle \langle t_{pq},w_{kl} \rangle q(v_{ij})- \langle f_i,x_p \rangle\langle f,x_k \rangle\langle t_{pq},w_{kl} \rangle q(v_{ij})  \\
				 &\hspace{1cm}\Bigl.+ \langle x,f_i \rangle q(v_{ij})\Bigr\}\langle f_i,x_k \rangle\langle  t_{pq},w_{kl} \rangle x_p  +\Bigl\{\langle f_i,x_k \rangle\langle f,x_p \rangle\langle t_{pq},w_{kl} \rangle q(v_{ij}) \Bigr.\\
				 &\hspace{1cm}\Bigl.+\langle v_{ij},z \rangle+\langle f_i,x \rangle q(v_{ij})
				 -\langle f,x_k \rangle\langle f_i,x_p\rangle \langle t_{pq},w_{kl}\rangle q(v_{ij})\Bigr\}  \langle f_i,x_p\rangle\langle t_{pq},w_{kl} \rangle x_k,\;\\
				 &\hspace{1cm} \left.f+\Bigl\{\langle f,x_p \rangle\langle f_i,x_k\rangle-\langle f,x_k \rangle\langle f_i,x_p\rangle      \Bigr\}\langle t_{pq},w_{kl}\rangle q(v_{ij}) f_i\right). \numberthis{eqnt1}  
    \end{align*}
    Now, for $\lambda_{kj}\;=\;\alpha_{kl}\delta_{pq}^*\beta_{ij}$ as in
the statement, we can describe the maps $\lambda_{kj},
\lambda_{kj}^*, \frac{1}{2}\lambda_{kj}\lambda_{kj}^* $ and the
elementary transformation $E_{\lambda_{kj}}$ as
\begin{align*}
       \lambda_{kj}(z,x,f)
       &= \alpha_{kl}\delta_{pq}^*\beta_{ij}(z,x,f)
        =\Bigl(0, \langle  v_{ij},z \rangle\langle f_i,x_p \rangle \langle  w_{kl},t_{pq} \rangle x_k,0 \Bigr), \\
        \lambda_{kj}^*(z,x,f)
      &= \beta_{ij}^* \delta_{pq} \alpha_{kl}^* (z,x,f)
      = \Bigl(\langle f,x_k \rangle \langle  t_{pq},w_{kl}\rangle \langle x_p,f_i \rangle v_{ij},0,0 \Bigr),\\
        \frac{1}{2}\lambda_{kj}\lambda_{kj}^*(z,x,f)
       &= \Bigl(0,\langle f,x_k \rangle \langle f_i,x_p \rangle^2 \langle  w_{kl},t_{pq} \rangle^2 q(v_{ij})  x_k,0\Bigr),\\
              E_{\lambda_{kj}}(z,x,f)
            &= \Bigl( I + \lambda_{kj} - \lambda_{kj}^* - \frac{1}{2}\lambda_{kj}\lambda_{kj}^* \Bigr)      (z,x,f)\\
            &=\Bigl(z -\langle f,x_ k\rangle \langle f_i,x_p \rangle \langle  w_{kl},t_{pq}                           \rangle v_{ij} , \Bigr.   x+\Bigl\{\langle v_{ij},z\rangle  \Bigr. \\
            &\quad\Bigl.-\langle f,x_k \rangle\langle f_i,x_p \rangle \langle  w_{kl},t_{pq} \rangle q(v_{ij})\Bigr\}x_k ,f \Bigr).
     \end{align*}
     If $i \neq k$, then, by Remark~\ref{r01}, we have
     \begin{align*}
        \left[E_{\beta_{ij}}^*,\right.\left.E_{\frac{\lambda_{kj}}{2}} \right] (z,x,f)
            &=  {\left[E_{\frac{\lambda_{kj}}{2}}, E_{\beta_{ij}}^* \right]}^{-1} (z,x,f)
            = \Bigl( I  - \frac{1}{2} \beta_{ij} \lambda_{kj}^* + \frac{1}{2} \lambda_{kj} \beta_{ij}^* \Bigr)(z,x,f)\\
            &= \Bigl(z,\; x+\langle x,f_i \rangle\langle f_i,x_p \rangle \langle  w_{kl},t_{pq} \rangle q(v_{ij}) x_k,\;
             f\Bigr.\\
             &\hspace{1cm}\Bigl.-\langle f,x_k \rangle\langle f_i,x_p \rangle \langle  w_{kl},t_{pq} \rangle q(v_{ij}) f_i\Bigr)
    \end{align*}
    and hence we get
    \begin{align*}
          E_{\lambda_{kj}}&\left[E_{\beta_{ij}}^*,E_\frac{\lambda_{kj}}{2} \right] (z,x,f)\\
           &= \Bigl( I + \lambda_{kj} - \lambda_{kj}^* - \frac{1}{2}\lambda_{kj}\lambda_{kj}^* - \frac{1}{2} \beta_{ij} \lambda_{kj}^* + \frac{1}{2} \lambda_{kj} \beta_{ij}^*   \Bigr)(z,x,f)\\
           &= \Bigl(z-\langle f,x_k\rangle\langle x_p,f_i \rangle\langle t_{pq},w_{kl} \rangle v_{ij}, x+\Bigl\{ \langle f_i,x \rangle  q(v_{ij}) +\langle v_{ij},z \rangle \Bigr.\Bigr.\\
           &\hspace{1cm}-\langle f,x_k \rangle\langle f_i,x_p\rangle\langle t_{pq},w_{kl}\rangle q(v_{ij})\Bigl.\Bigr\}\langle t_{pq},w_{kl} \rangle \langle f_i,x_p\rangle x_k,\\
           &\hspace{1cm}\Bigl.f-\langle f,x_k \rangle\langle f_i,x_p\rangle\langle t_{pq},w_{kl}\rangle q(v_{ij})f_i\Bigr). \numberthis{eqnt3}
      \end{align*}
      Similarly, if $i \neq p$, we have
      \begin{align*}
          E_{\xi_{pj}}&\left[E_{\beta_{ij}}^*,E_\frac{\xi_{pj}}{2} \right] (z,x,f)\\
           &= \Bigl ( I + \xi_{pj} - \xi_{pj}^* - \frac{1}{2}\xi_{pj}\xi_{pj}^* - \frac{1}{2} \beta_{ij} \xi_{pj}^* + \frac{1}{2} \xi_{pj} \beta_{ij}^*   \Bigr)(z,x,f)\\
           &= \Bigl(z-\langle f,x_k\rangle\langle x_p,f_i \rangle\langle t_{pq},w_{kl} \rangle v_{ij},x +\Bigl\{\langle f_i,x \rangle q(v_{ij})\Bigr.\Bigr.\\
       &\hspace{1cm} -\langle f,x_k \rangle\langle f_i,x_p\rangle \langle t_{pq},w_{kl}\rangle q(v_{ij})\Bigl.+\langle v_{ij},z \rangle\Bigr\}\langle t_{pq},w_{kl}\rangle\\
       &\hspace{1cm}\langle f_i,x_p\rangle  x_k, \Bigl. f-\langle f,x_k \rangle\langle f_i,x_p\rangle\langle t_{pq},w_{kl}\rangle q(v_{ij})f_i\Bigr    ). \numberthis{eqnt4}
      \end{align*}

     We now consider the following possible conditions on the indices.

     \smallskip

\noindent\textbf{Case (i):} $i = p$.

  \smallskip

      \noindent If $i = p$, then, by Equations~\eqref{eqnt1},~\eqref{eqnt2}, and~\eqref{eqnt3}, we have
    \begin{align*}
    \left[E_{\beta_{ij}}^*\right.,&\left[E_{\alpha_{kl}},E_{\delta_{pq}}\left.\right]\right](z,x,f)\\
          &= \Bigl( I - \beta_{pj}^* \delta_{pq} \alpha_{kl}^* + \alpha_{kl} \delta_{pq}^*\beta_{pj}  + \frac{1}{2} \alpha_{kl} \delta_{pq}^*\beta_{pj}\beta_{pj}^* - \frac{1}{2}\beta_{pj}\beta_{pj}^*\delta_{pq} \alpha_{kl}^* \Bigr.\\
          &\hspace{1cm}\left.- \frac{1}{2} \alpha_{kl} \delta_{pq}^*\beta_{pj}\beta_{pj}^* \delta_{pq} \alpha_{kl}^*\right)(z,x,f)\\
          &= \Bigl(z-\langle f,x_k\rangle\langle t_{pq},w_{kl} \rangle v_{pj},\;x+ \Bigl\{\langle v_{pj},z \rangle - \langle f,x_k \rangle\langle t_{pq},w_{kl}\rangle q(v_{pj}) \Bigr.\\
          &\hspace{1cm} \Bigl.\Bigl.+ \langle f_p,x \rangle q(v_{pj})\Bigr\}\langle  t_{pq},w_{kl} \rangle x_k,\;f-\langle f,x_k \rangle\langle                         t_{pq},w_{kl}\rangle q(v_{pj})f_p \Bigr)\\
          =\;& E_{\lambda_{kj}}\left[ E_{\beta_{ij}}^*,E_\frac{\lambda_{kj}}{2} \right] (z,x,f).
    \end{align*}
    \noindent \textbf{Case (ii):} $i = k.$

    \smallskip
    \noindent If $i = k$, then, by Equations~\eqref{eqnt1},~\eqref{eqnt2}, and~\eqref{eqnt4}, we have
    \begin{align*}
    \left[ E_{\beta_{ij}}^*\right.,&\left[ E_{\alpha_{kl}},E_{\delta_{pq}} \right]\left. \right](z,x,f)\\
          &= \Bigl( I + \beta_{kj}^* \alpha_{kl} \delta_{pq}^*   + \frac{1}{2} \beta_{kj}\beta_{kj}^*\alpha_{kl}\delta_{pq}^* - \frac{1}{2}\delta_{pq}\alpha_{kl}^*\beta_{kj}\beta_{kj}^* - \delta_{pq}\alpha_{kl}^*\beta_{kj}\Bigr.\\
          &\hspace{1cm}\Bigl.- \frac{1}{2} \delta_{pq}\alpha_{kl}^*\beta_{kj}\beta_{kj}^* \alpha_{kl} \delta_{pq}^* \Bigr)(z,x,f)\\
          &= \Bigl(z+\langle f,x_p\rangle\langle t_{pq},w_{kl} \rangle v_{kj}, x  - \Bigl\{   \langle v_{kj},z \rangle + \langle f,x_p \rangle \langle t_{pq},w_{kl} \rangle q(v_{kj}) \Bigr.\Bigr. \\
         & \hspace{1cm}\Bigl. +\langle x,f_k \rangle q(v_{kj})\Bigr\}\langle  t_{pq},w_{kl} \rangle x_p ,\Bigl. f+ \langle f,x_p \rangle\langle t_{pq},w_{kl}\rangle q(v_{kj})  f_k \Bigr)\\
          &= E_{\xi_{pj}}\left[ E_{\beta_{kj}}^*,E_\frac{\xi_{pj}}{2} \right] (z,x,f).
    \end{align*}
      \noindent \textbf{Case(iii): } $i \neq k$ and $i \neq p$.

    \smallskip

    \noindent If $i \neq k$ and $i \neq p$, then, by Equation~\eqref{eqnt1}, we have
    \[
    \left[ E_{\beta_{ij}}^*,\left[ E_{\alpha_{kl}},E_{\delta_{pq}} \right] \right](z,x,f) = I(z,x,f).\qedhere
    \]
\end{proof}

As a consequence of the above lemma on triple commutators, we
observe the following commutator relations.

\begin{cor}\label{c04}
    For any $i,j,k,l,p,q$ with $1 \leq i,k,p \leq m$, $1 \leq j,l,q \leq n$, $i \neq k$ and $k \neq p$ and $a,b,c,d,e,f \in A$ with $abc=def$ and $a^2bc=d^2ef$, the following equation holds.
    \[
     \left[ E_{a\beta_{ij}}^*, \left[ E_{b\alpha_{kl}},E_{c\delta_{pq}} \right] \right] =  \left[ E_{d\beta_{ij}}^*,\left[ E_{e\alpha_{kl}}, E_{f\delta_{pq}} \right] \right] .
    \]
\end{cor}
\begin{proof}
For any $i,j,k,l,p,q$ with $1 \leq i,k,p \leq m$, $1 \leq j,l,q \leq n$, $i \neq k$ and $k \neq p$ and $a,b,c,d,e,f \in A$ with $abc=def$ and $a^2bc=d^2ef$, we have
\begin{align*}
 \left[ E_{a\beta_{ij}}^*, \left[ E_{b\alpha_{kl}},E_{c\delta_{pq}} \right] \right](z,x,f)
          &= \Bigl( I - a^2bc\beta_{ij}^* \delta_{pq} \alpha_{kl}^* + abc\alpha_{kl}                    \delta_{pq}^*\beta_{ij}  + \frac{1}{2}  a^2bc\alpha_{kl}                        \delta_{pq}^*\beta_{ij}\beta_{ij}^*  \Bigr.\\
          &\quad\Bigl. - \frac{1}{2}a^2bc\beta_{ij}\beta_{ij}^*\delta_{pq} \alpha_{kl}^* - \frac{1}{2}              a^2b^2c^2\alpha_{kl} \delta_{pq}^*\beta_{ij}\beta_{ij}^* \delta_{pq}                    \alpha_{kl}^*\Bigr)(z,x,f)\\
          &=\Bigl( I - d^2ef\beta_{ij}^* \delta_{pq} \alpha_{kl}^* + def\alpha_{kl} \delta_{pq}^*\beta_{ij}
              + \frac{1}{2} d^2ef\alpha_{kl} \delta_{pq}^*\beta_{ij}\beta_{ij}^*  \Bigr.\\
        &\Bigl. - \frac{1}{2}d^2ef\beta_{ij}\beta_{ij}^*\delta_{pq} \alpha_{kl}^* - \frac{1}{2} d^2e^2f^2\alpha_{kl} \delta_{pq}^*\beta_{ij}\beta_{ij}^* \delta_{pq} \alpha_{kl}^*\Bigr)(z,x,f)\\
 =\;& \left[ E_{d\beta_{ij}}^*, \left[ E_{e\alpha_{kl}},E_{f\delta_{pq}} \right] \right](z,x,f).\qedhere
\end{align*}
\end{proof}
The following lemma on triple commutators involves a mixed
commutator.

\begin{lem}\label{l05}
    Let $\alpha, \delta \in \hom_A(Q,P)$ and $\beta \in \hom_A(Q,P^*)$.
    Then, for $i,j,k,l,p,q$ with $1 \leq i,k,p \leq m$, $1 \leq j,l,q \leq n$
    and $k \neq p$, the triple commutator
    $\left[
        E_{\alpha_{ij}},\left[ E_{\delta_{kl}},E_{\beta_{pq}}^*\right]
    \right]$ is given by
   \[
     \left[E_{\alpha_{ij}},\left[ E_{\delta_{kl}},E_{\beta_{pq}}^* \right]\right] =
     \begin{cases}
      \vspace{2mm}
       E_{\mu_{kj}}\left[ E_{\alpha_{ij}},E_\frac{\mu_{kj}}{2} \right], &\textnormal{ if }\quad i = p,\\
       I         &\textnormal{ if }\quad i = k \quad\textnormal{ or }\quad i \neq p,
       \end{cases}
   \]
    where $\mu_{kj}=\delta_{kl}\beta_{pq}^*\alpha_{ij}$.
\end{lem}
\begin{proof}
For $\alpha, \delta \in \hom_A(Q,P)$ and  $\beta \in \hom_A(Q,P^*)$ and for any $i,j,k,l,p,q$ with $1 \leq i,k,p \leq m$, $1 \leq j,l,q \leq n$ and $k \neq p$, we have the coordinate-free expressions
\begin{align*}
        \left[ E_{\alpha_{ij}}, \right.&\left[ E_{\delta_{kl}},E_{\beta_{pq}}^* \right]\left.\right](z,x,f)\\
        &= \Bigl( I +\delta_{kl} \beta_{pq}^* \alpha_{ij} + \alpha_{ij}^* \beta_{pq}\delta_{kl}^* \beta_{pq}\delta_{kl}^*-  \beta_{pq}\delta_{kl}^*\beta_{pq}\delta_{kl}^*    \Bigr.  - \delta_{kl} \beta_{pq}^*\delta_{kl} \beta_{pq}^*- \frac{1}{2} \delta_{kl} \beta_{pq}^* \alpha_{ij} \alpha_{ij}^* \beta_{pq}\delta_{kl}^* \\
        &\hspace{1cm}- \alpha_{ij}^* \beta_{pq}\delta_{kl}^*  + \frac{1}{2}\alpha_{ij} \alpha_{ij}^* \beta_{pq}\delta_{kl}^* \beta_{pq}\delta_{kl}^* - \frac{1}{2}\alpha_{ij} \alpha_{ij}^* \beta_{pq}\delta_{kl}^*+ \frac{1}{2} \delta_{kl} \beta_{pq}^* \alpha_{ij} \alpha_{ij}^*\Bigr)(z,x,f)\Bigr). \numberthis{eqnt6}
   \end{align*}
   Now if we use coordinates, we obtain
\begin{align*}
 \left[ E_{\alpha_{ij}},\right.& \left[ E_{\delta_{kl}},E_{\beta_{pq}}^* \right]\left.\right](z,x,f)\\
 &= \Bigl(z-\langle f,x_k \rangle \langle t_{kl}, v_{pq} \rangle \langle f_p,x_i \rangle w_{ij},\Bigr. x+ \Bigl\{ \langle f,x_i \rangle q(w_{ij}) +\langle w_{ij},z \rangle\Bigr.\\
          &\hspace{1cm}\Bigl.  - \langle f,x_k \rangle  \langle t_{kl}, v_{pq} \rangle  \langle x_i,f_p \rangle  q(w_{ij}) \Bigr\}\langle t_{kl},v_{pq}\rangle \langle x_i,f_p \rangle x_k\\
          &\hspace{1cm} - \langle f,x_k \rangle \langle t_{kl}, v_{pq} \rangle \langle x_i,f_p \rangle q(w_{ij}) x_i,\Bigl. f\Bigr). \numberthis{eqnt5}
  \end{align*}
   The maps $\mu_{kj}$, $\mu_{kj}^*,
   \frac{1}{2}\mu_{kj}\mu_{kj}^* $ and the elementary transformation $E_{\mu_{kj}}^*$ are
   given by the following expressions.
\begin{align*}
      \mu_{kj}(z,x,f)    =\;& \delta_{kl}\beta_{pq}^*\alpha_{ij}(z,x,f)
               = (0,\langle w_{ij},z \rangle \langle t_{kl},v_{pq}\rangle \langle x_i,f_p \rangle x_k,0),\\
      \mu_{kj}^*(z,x,f) =\;& \alpha_{ij}^*\beta_{pq}\delta_{kl}^*(z,x,f)
               = (\langle f,x_k \rangle \langle t_{kl}, v_{pq} \rangle \langle x_i,f_p \rangle w_{ij},0,0),\\
      \frac{1}{2}\mu_{kj}\mu_{kj}^*(z,x,f)
             =\;& (0,\langle f,x_k \rangle \langle t_{kl}, v_{pq} \rangle^2 \langle x_i,f_p \rangle^2  q(w_{ij})x_k,0),\\
      E_{\mu_{kj}}^*(z,x,f)
            =\;& \left( I + \mu_{kj} - \mu_{kj}^* - \frac{1}{2}\mu_{kj}\mu_{kj}^* \right)(z,x,f)\\
            =\;& (z - \langle f,x_k \rangle \langle t_{kl}, v_{pq} \rangle \langle x_i,f_p \rangle w_{ij}, x +\langle w_{ij},z \rangle \langle    t_{kl},v_{pq}\rangle \langle x_i,f_p \rangle x_k \\
            &- \langle f,x_k \rangle \langle t_{kl}, v_{pq} \rangle^2 \langle x_i,f_p \rangle^2 q(w_{ij}) x_k,\;f\;).
    \end{align*}
    If $i \neq k$, then, by Lemma~\ref{l01}, we have
    \begin{align*}
        \left[ E_{\alpha_{ij}},\right.&\left.E_\frac{\mu_{kj}}{2} \right] (z,x,f)\\
              &= \Bigl( I + \frac{1}{2}\mu_{kj}\alpha_{ij}^* - \frac{1}{2} \alpha_{ij} \mu_{kj}^*\Bigr)(z,x,f)\\
              &= \Bigl(z, x + \langle f,x_i \rangle \langle t_{kl},v_{pq}\rangle \langle
              x_i,f_p \rangle q(w_{ij})x_k - \langle f,x_k \rangle \langle t_{kl}, v_{pq}
              \rangle\langle x_i,f_p \rangle q(w_{ij}) x_i,f \Bigr)
                  \end{align*}
    and hence we get
    \begin{align*}
       E_{\mu_{kj}}&\left[ E_{\alpha_{ij}},E_\frac{\mu_{kj}}{2} \right] (z,x,f)\\
              &= \Bigl(I + \mu_{kj} - \mu_{kj}^* - \frac{1}{2}\mu_{kj}\mu_{kj}^* + \frac{1}{2}\mu_{kj}\alpha_{ij}^* - \frac{1}{2} \alpha_{ij} \mu_{kj}^* \Bigr) (z,x,f)\\
               &= \Bigl(z - \langle f,x_k \rangle \langle t_{kl}, v_{pq} \rangle \langle x_i,f_p \rangle w_{ij},\Bigr.  x + \Bigl\{ \langle f,x_i \rangle q(w_{ij})+\langle w_{ij},z \rangle  \\
              &\hspace{1cm}- \langle f,x_k \rangle \langle t_{kl}, v_{pq} \rangle \langle x_i,f_p \rangle q(w_{ij})\Bigr\} \langle t_{kl}, v_{pq} \rangle \langle x_i,f_p \rangle x_k \\
              &\hspace{1cm} \Bigl.\Bigl.- \langle f,x_k \rangle \langle t_{kl}, v_{pq} \rangle\langle x_i,f_p \rangle q(w_{ij}) x_i,\;f\Bigr). \numberthis{eqnt7}
\end{align*}

We now consider the following possible conditions on the indices.

     \smallskip

\noindent\textbf{Case(i):} $i = p$.

  \smallskip

      \noindent If $i = p$, then, by Equations~\eqref{eqnt5}, \eqref{eqnt6} and \eqref{eqnt7}, we have
  \begin{align*}
      \left[ E_{\alpha_{ij}},\right.&\Bigl[ E_{\delta_{kl}},E_{\beta_{pq}}^* \Bigr]\left.\right](z,x,f)\\
        &= \Bigl( I - \alpha_{ij}^* \beta_{pq}\delta_{kl}^* + \delta_{kl} \beta_{pq}^* \alpha_{ij}  + \frac{1}{2} \delta_{kl} \beta_{pq}^* \alpha_{ij} \alpha_{ij}^* - \frac{1}{2} \delta_{kl} \beta_{pq}^* \alpha_{ij} \alpha_{ij}^* \beta_{pq}\delta_{kl}^*\Bigr. \\
        &\hspace{1cm}\left. \left.  - \frac{1}{2}\alpha_{ij} \alpha_{ij}^* \beta_{pq}\delta_{kl}^*  \right.\right.\Bigl. + \alpha_{ij}^* \beta_{pq}\delta_{kl}^* \beta_{pq}\delta_{kl}^*  \Bigr)(z,x,f) \\
        &= \Bigl(z-\langle f,x_k \rangle \langle t_{kl}, v_{pq} \rangle w_{ij},\;x + \langle w_{ij},z \rangle \langle t_{kl},v_{pq}\rangle x_k + \langle f,x_i \rangle \langle t_{kl},v_{pq}\rangle q(w_{ij}) x_k\Bigr.\\
        &\hspace{1cm}
         - \langle f,x_k \rangle \langle t_{kl}, v_{pq} \rangle q(w_{ij}) x_i  \Bigl.- \langle f,x_k \rangle \langle t_{kl}, v_{pq} \rangle^2 q(w_{ij}) x_k,f\Bigr)\\
        &= E_{\mu_{kj}}\left[ E_{\alpha_{ij}},E_\frac{\mu_{kj}}{2}\right] (z,x,f).
\end{align*}
\noindent \textbf{Case(ii): } $i = k$ or $i \neq p$.

    \smallskip

    \noindent If $i = k$ or $i \neq p$, then, by Equation~\eqref{eqnt5}, we have
    \[
    \left[ E_{\alpha_{ij}},\left[ E_{\delta_{kl}},E_{\beta_{pq}}^* \right]\right](z,x,f) = I(z,x,f).\qedhere
    \]
\end{proof}

We now deduce the commutator identities from the above lemma.

\begin{cor}\label{c05}
 For any $i,j,k,l,p,q$ with $1 \leq i,k,p \leq m$, $1 \leq j,l,q \leq n$, $i \neq p$ and $k \neq p$ and $a,b,c,d,e,f \in A$ with $abc=def$ and $a^2bc=d^2ef$, the following equation holds.
 \[
 \left[ E_{a\alpha_{ij}}, \left[ E_{b\delta_{kl}},E_{c\beta_{pq}}^* \right] \right] = \left[ E_{d\alpha_{ij}}, \left[ E_{e\delta_{kl}}, E_{f\beta_{pq}}^* \right] \right].
 \]
\end{cor}

\begin{lem}\label{l06}
Let $\alpha \in \hom_A(Q,P)$ and $\beta, \gamma \in \hom_A(Q,P^*)$. Then, for $i,j,k,l,p,q$ with $1 \leq i,k,p \leq m$, $1 \leq j,l,q \leq n$ and $k \neq p$, the triple commutator $\left[E_{\beta_{ij}}^*,\left[ E_{\alpha_{kl}}, E_{\gamma_{pq}}^* \right] \right]$
is given by
\[
    \left[E_{\beta_{ij}}^*,\left[ E_{\alpha_{kl}}, E_{\gamma_{pq}}^* \right] \right]
      = \begin{cases}
        \vspace{2mm}
        E_{\nu_{pj}}^* \left[ E_{\beta_{ij}}^*,E_\frac{\nu_{pj}}{2}^* \right],
         &\textnormal{ if }\quad i = p,\\
        I   &\textnormal{ if }\quad i = k \quad\textnormal{ or }\quad i \neq p,
       \end{cases}
   \]
    where $\nu_{pj}= -\gamma_{pq}\alpha_{kl}^*\beta_{ij}$.
 \end{lem}
\begin{proof}
For $\alpha \in \hom_A(Q,P)$, $\beta, \gamma \in \hom_A(Q,P^*)$ and, for $i,j,k,l,p,q$
with $1 \leq i,k,p \leq m$, $1 \leq j,l,q \leq n$ and $k \neq p$, we have the following coordinate-free expression.
\begin{align*}
\left[E_{\beta_{ij}}^* ,\right.&\left. \left[ E_{\alpha_{kl}}, E_{\gamma_{pq}}^* \right]\right](z,x,f)\\
            &= E_{\beta_{ij}}^* \left[ E_{\alpha_{kl}}, E_{\gamma_{pq}}^* \right] E_{\beta_{ij}}^{*^{-1}}
                {\left[ E_{\alpha_{kl}}, E_{\gamma_{pq}}^* \right]}^{-1}(z,x,f) \\
            &=  \Bigl( I  +  \beta_{ij}^* \alpha_{kl}\gamma_{pq}^*  - \gamma_{pq} \alpha_{kl}^* \beta_{ij}   - \frac{1}{2}\gamma_{pq} \alpha_{kl}^* \beta_{ij}\beta_{ij}^* + \frac{1}{2}\beta_{ij}\beta_{ij}^*\alpha_{kl}\gamma_{pq}^* \Bigr.\\
            &\hspace{1cm} \Bigl.  - \frac{1}{2}\gamma_{pq} \alpha_{kl}^*\beta_{ij}\beta_{ij}^*\alpha_{kl}\gamma_{pq}^* \Bigr)(z,x,f) \numberthis{eqnt9}
\end{align*}
Now, by computing using coordinates, we have
\begin{align*}
\left[E_{\beta_{ij}}^* ,\right.&\left. \left[E_{\alpha_{kl}}, E_{\gamma_{pq}}^*\right]\right]\\
            &= \Bigl(z +\langle x,f_p \rangle \langle c_{pq},w_{kl}\rangle \langle x_k,f_i \rangle v_{ij},\; x,\;f-
            \Bigl\{ \langle x,f_p \rangle \langle c_{pq},w_{kl}\rangle \langle x_k,f_i \rangle q(v_{ij}) + \langle  v_{ij},z \rangle   \Bigr.\Bigr.\\
            &\hspace{1cm} \Bigl. 
           +\langle x,f_i \rangle q(v_{ij})  \Bigr\}\langle c_{pq},w_{kl}\rangle \langle x_k,f_i \rangle f_p           \Bigl. +\langle x,f_p \rangle \langle c_{pq},w_{kl}\rangle \langle x_k,f_i \rangle q(v_{ij})f_i\Bigr).
           \numberthis{eqnt8}
\end{align*}
The maps $\nu_{pj}$ in the statement of the lemma, as well as the
other maps $\nu_{pj}^*, \frac{1}{2}\nu_{pj}\nu_{pj}^*$ and the
transformations $E_{\nu_{pj}}^*$ are given as
\begin{align*}
\nu_{pj}(z,x,f) = & -\gamma_{pq}\alpha_{kl}^*\beta_{ij}(z,x,f) = \Bigl(0,0,-\langle v_{ij},z \rangle \langle c_{pq},w_{kl}\rangle \langle f_i,x_p \rangle f_k\Bigr),\\
\nu_{pj}^*(z,x,f) = & -\beta_{ij}^*\alpha_{pq}\gamma_{kl}^*(z,x,f) = \Bigl(-\langle x,f_p \rangle \langle c_{pq}, w_{kl} \rangle \langle f_i,x_k \rangle v_{ij},0,0 \Bigr),\\
\frac{1}{2}\nu_{pj}\nu_{pj}^*(z,x,f) = &\Bigl(0,0,\langle x,f_p \rangle
\langle c_{pq}, w_{kl} \rangle^2 \langle f_i,x_k \rangle^2 q(v_{ij})
f_p\Bigr),\\
E_{\nu_{pj}}^*(z,x,f)
      =\;& \Bigl( I + \nu_{pj} - \nu_{pj}^* - \frac{1}{2}\nu_{pj} \nu_{pj}^* \Bigr)(z,x,f)\\
      =\;& \Bigl(z + \langle x,f_p \rangle \langle c_{pq}, w_{kl} \rangle \langle x_k,f_i \rangle v_{ij},\;x ,\; f - \langle v_{ij},z \rangle \langle c_{pq},w_{kl}\rangle \langle x_k,f_i \rangle f_p\Bigr.\\
      &\Bigl. - \langle x,f_p \rangle \langle c_{pq}, w_{kl} \rangle^2 \langle x_k,f_i \rangle^2 q(v_{ij}) f_p\Bigr).
\end{align*}
If $i\neq p$, then, by Lemma~\ref{l03}, we have
\begin{align*}
    \left[E_{\beta_{ij}}^*,E_\frac{\nu_{pj}}{2}^* \right] (z,x,f)
          =\;& \Bigl(I + \frac{1}{2} \nu_{pj} \beta_{ij}^* - \frac{1}{2} \beta_{ij} \nu_{pj}^* \Bigr)(z,x,f)\\
          =\;&\Bigl(z,\; x, \;f + \langle x,f_p \rangle \langle c_{pq}, w_{kl} \rangle\langle x_k,f_i \rangle q(v_{ij}) f_i\\
          \;&\;- \langle x,f_i \rangle \langle c_{pq},w_{kl}\rangle \langle x_k,f_i \rangle q(v_{ij})f_p\Bigr)
\end{align*}
and hence we get
\begin{align*}
 E_{\nu_{pj}}^* \left[ E_{\beta_{ij}}^*,E_\frac{\nu_{pj}}{2}^* \right] (z,x,f)
 &=\Bigl( I + \nu_{pj} - \nu_{pj}^* - \frac{1}{2}\nu_{pj} \nu_{pj}^*+ \frac{1}{2} \nu_{pj} \beta_{ij}^* - \frac{1}{2} \beta_{ij} \nu_{pj}^* \Bigr)(z,x,f)\\
 &= \Bigl( z + \langle x,f_p \rangle \langle c_{pq},w_{kl}\rangle \langle x_k,f_i \rangle v_{ij},\; x, f- \Bigl\{ \langle x,f_i \rangle  q(v_{ij}\Bigr.\Bigr.  \\
 &\quad  +  \langle  v_{ij},z \rangle+\langle x,f_p \rangle \langle c_{pq},w_{kl}\rangle \langle x_k,f_i \rangle q(v_{ij}) \Bigr\}\langle c_{pq},w_{kl}\rangle  \\
        &\quad \langle x_k,f_i \rangle f_p + \langle x,f_p \rangle \langle c_{pq},w_{kl}\rangle \langle x_k,f_i \rangle q(v_{ij})f_i\Bigr).
 \numberthis{eqnt10}
\end{align*}

\medskip

We now consider the following possible conditions on the indices.

     \smallskip

\noindent\textbf{Case(i):} $i = k$.

  \smallskip

      \noindent If $i = k$, then, by Equations~\eqref{eqnt8}, \eqref{eqnt9} and \eqref{eqnt10}, we have
  \begin{align*}
      \left[E_{\beta_{ij}}^* ,\right.& \left.\left[ E_{\alpha_{kl}}, E_{\gamma_{pq}}^* \right]\right](z,x,f)\\
    &=  \Bigl( I  +  \beta_{ij}^* \alpha_{kl}\gamma_{pq}^*  - \gamma_{pq} \alpha_{kl}^* \beta_{ij}  - \frac{1}{2}\gamma_{pq} \alpha_{kl}^* \beta_{ij}\beta_{ij}^*+ \frac{1}{2}\beta_{ij}\beta_{ij}^*\alpha_{kl}\gamma_{pq}^* \Bigr.\\
    & \hspace{1cm}  \Bigl.   - \frac{1}{2}\gamma_{pq} \alpha_{kl}^*\beta_{ij}\beta_{ij}^*\alpha_{kl}\gamma_{pq}^* \Bigr)(z,x,f)\\
    &=  \Bigl( z + \langle x,f_p \rangle \langle c_{pq},w_{kl}\rangle  v_{ij},\; x,\;f - \langle  v_{ij},z \rangle \langle c_{pq},w_{kl}\rangle f_p+\langle x,f_p \rangle \langle c_{pq},w_{kl}\rangle  q(v_{ij})f_i\bigr.\\
    &\hspace{1cm}  -\langle x,f_p \rangle \langle c_{pq},w_{kl}\rangle^2  q(v_{ij})f_p \Bigl. - \langle x,f_i \rangle \langle c_{pq},w_{kl}\rangle  q(v_{ij})f_p\Bigr)\\
    &= E_{\nu_{pj}}^* \left[ E_{\beta_{ij}}^*,E_\frac{\nu_{pj}}{2}^* \right] (z,x,f).
\end{align*}
\noindent \textbf{Case(ii):} $i = p$ or $i \neq k$.

    \smallskip

    \noindent If $i = k$ or $i \neq p$, then, by Equation~\eqref{eqnt5}, we have
    \[
    \left[E_{\beta_{ij}}^* , \left[ E_{\alpha_{kl}}, E_{\gamma_{pq}}^* \right]\right](z,x,f) = I(z,x,f).\qedhere
    \]
\end{proof}
The set of commutator relations we deduce from the above lemma is given in
the corollary below.

\begin{cor}\label{c06}
 For any given $i,j,k,l,p,q$, where $1 \leq i,k,p \leq m$, $1 \leq j,l,q \leq n$ 
such that $i \neq k$ and $k \neq p$ and $a,b,c,d,e,f \in A$, $\left[ E_{a\beta_{ij}}^*,\left[ E_{b\gamma_{kl}}^*,E_{c\alpha_{pq}} \right]\right]= \left[ E_{d\beta_{ij}}^*,\left[ E_{e\gamma_{kl}}^*,E_{f\alpha_{pq}} \right]\right]$ if $abc=def$ and $a^2bc=d^2ef$.
\end{cor}
Finally, another triple commutator is computed in the following lemma and the commutator relations which follow from this are stated in the corollary below this lemma.

\begin{lem}\label{l07}
Let $\alpha \in \hom_A(Q,P)$ and $\beta, \gamma \in \hom_A(Q,P^*)$. Then, for $i,j,k,l,p,q$
with $1 \leq i,k,p \leq m$, $1 \leq j,l,q \leq n$ and $k \neq p$, the triple
commutator $\left[ E_{\alpha_{ij}},\left[ E_{\beta_{kl}}^*,E_{\gamma_{pq}}^* \right] \right]$
is given by
\[
          \left[ E_{\alpha_{ij}},\left[ E_{\beta_{kl}}^*,E_{\gamma_{pq}}^* \right]\right]
          =     \begin{cases}
            \vspace{2mm}
            E_{\eta_{kj}}^* \left[E_{\alpha_{ij}},E_\frac{\eta_{kj}}{2}^*\right]
            &\textnormal{if}\quad i = p,\\
            E_{\vartheta_{pj}}^*\left[ E_{\alpha_{ij}},E_\frac{\vartheta_{pj}}{2}^* \right]
            &\textnormal{if}\quad i = k,\\
            I & \textnormal{if}\quad  i \neq p \textnormal{ and } i \neq k,
             \end{cases}
         \]
 where $\eta_{kj}=\beta_{kl}\gamma_{pq}^*\alpha_{ij}$ and $\vartheta_{pj} = \gamma_{pq}\beta_{kl}^*\alpha_{ij}$.
\end{lem}
\begin{proof}
For $\alpha \in \hom_A(Q,P)$, $\beta, \gamma \in \hom_A(Q,P^*)$ and for $i,j,k,l,p,q$ with
$1 \leq i,k,p \leq m$, $1 \leq j,l,q \leq n$ and $k \neq p$, we have
\begin{align*}
 \left[ E_{\beta_{kl}}^*, E_{\gamma_{pq}}^* \right](z,x,f)
 &= \Bigl(I+\gamma_{pq}\beta_{kl}^*-\beta_{kl}\gamma_{kl}^* \Bigr)(z,x,f)\\
 &= \Bigl( z,\;x,\;f+\langle x,f_k \rangle \langle c_{pq},v_{kl} \rangle f_p - \langle x,f_p \rangle \langle v_{kl}, c_{pq}\rangle f_k \Bigr).\\
& \hspace{6cm}(\textnormal{by Lemma}\;\ref{l03})\\
\left[ E_{\beta_{kl}}^*,E_{\gamma_{pq}}^* \right]^{-1}(z,x,f)
&= \left[E_{\gamma_{pq}}^*,E_{\beta_{kl}}^* \right](z,x,f)
= \Bigl( I-\gamma_{pq}\beta_{kl}^*+\beta_{kl}\gamma_{pq}^* \Bigr)(z,x,f)\\
&= \Bigl(z,\;x,\;f-\langle x,f_k \rangle \langle c_{pq},v_{kl} \rangle f_p + \langle x,f_p \rangle \langle v_{kl}, c_{pq}\rangle f_k\Bigr).\\\
& \hspace{6cm}(\textnormal{by Lemma}\;\ref{l03})
\end{align*}
Hence we get
\begin{align*}
        \left[ E_{\alpha_{ij}},\right.&\left. \left[E_{\beta_{kl}}^*,E_{\gamma_{pq}}^* \right] \right] (z,x,f)\\ 
              &=\Bigl(I + \alpha_{ij}^*\beta_{kl}\gamma_{pq}^*
              -\alpha_{ij}^*\gamma_{pq}\beta_{kl}^*
              + \frac{1}{2}\gamma_{pq}\beta_{kl}^*\alpha_{ij}\alpha_{ij}^*\gamma_{pq}\beta_{kl}^*
              - \frac{1}{2}\gamma_{pq}\beta_{kl}^*\alpha_{ij}\alpha_{ij}^*\beta_{kl}\gamma_{pq}^*
              \Bigr.\\
              &\hspace{1cm}
              + \frac{1}{2}\alpha_{ij}\alpha_{ij}^*\beta_{kl}\gamma_{pq}^*
                   - \frac{1}{2}\alpha_{ij}\alpha_{ij}^*\gamma_{pq}\beta_{kl}^*
              +\beta_{kl}\gamma_{pq}^*\alpha_{ij}
              - \frac{1}{2}\gamma_{pq}\beta_{kl}^*\alpha_{ij}\alpha_{ij}^*-\gamma_{pq}\beta_{kl}^*\alpha_{ij}\\
              &\hspace{1cm} - \frac{1}{2}\beta_{kl}\gamma_{pq}^*\alpha_{ij}\alpha_{ij}^*\gamma_{pq}\beta_{kl}^*
              +\frac{1}{2}\beta_{kl}\gamma_{pq}^*\alpha_{ij}\alpha_{ij}^*
               \Bigl.
              +\frac{1}{2}\beta_{kl}\gamma_{pq}^*\alpha_{ij}\alpha_{ij}^*\beta_{kl}\gamma_{pq}^*
              \Bigr)(z,x,f).
             \numberthis{eqnt12}
\end{align*}
Computing with coordinates, we get
\begin{align*}
        \left[ E_{\alpha_{ij}},\right.&\left. \left[ E_{\beta_{kl}}^*,E_{\gamma_{pq}}^* \right] \right] (z,x,f)\\
              &= \Bigl(z +
                  \Bigl\{
                     \langle x,f_p \rangle \langle x_i,f_k \rangle
                     - \langle x,f_k \rangle \langle x_i,f_p \rangle
                  \Bigr\}\langle c_{pq},v_{kl} \rangle w_{ij},x+\Bigl\{
                      \langle x,f_p \rangle  \langle x_i,f_k \rangle
               \Bigr.\\
                 &\hspace{1cm}  \Bigl.
                       - \langle x,f_k \rangle  \langle x_i,f_p \rangle
                     \Bigr\}\langle c_{pq},v_{kl} \rangle   
               q(w_{ij})x_i,\; f
                +\Bigl\{
                \langle  w_{ij},z \rangle+ \langle f,x_i \rangle q(w_{ij})
                \Bigr.
              q(w_{ij})\\
                  &\hspace{1cm}
                  -\langle x,f_k \rangle \langle c_{pq},v_{kl} \rangle \langle x_i,f_p \rangle
                  -\langle x,f_p \rangle \langle c_{pq},v_{kl} \rangle \langle x_i,f_k \rangle  \Bigl. q(w_{ij}) \Bigr\}\langle x_i,f_p \rangle\langle c_{pq},v_{kl} \rangle f_k           
              \\
              &\hspace{1cm} +\Bigl\{
              \langle x,f_k \rangle \langle c_{pq},v_{kl} \rangle
                 \Bigr.
              \langle x_i,f_p \rangle q(w_{ij})
               -\langle  w_{ij},z \rangle
               -\langle x,f_p \rangle \langle c_{pq},v_{kl} \rangle  \langle x_i,f_k \rangle  q(w_{ij})\\
               &\hspace{1cm}\Bigl.\Bigl.-\langle f,x_i \rangle q(w_{ij})\Bigr\} \langle x_i,f_k \rangle \langle c_{pq},v_{kl} \rangle f_p
              \Bigr). \numberthis{eqnt11}
\end{align*}
The transformations $\eta_{kj},\eta_{kj}^*,\frac{1}{2}\eta_{kj}\eta_{kj}^*$ and $E_{\eta_{kj}}^*$ are given by
\begin{align*}
\eta_{kj}(z,x,f) = &\;\beta_{kl}\gamma_{pq}^*\alpha_{ij}(z,x,f) = \Bigl(0,0,\langle w_{ij},z \rangle \langle c_{pq},v_{kl}\rangle \langle f_p,x_i \rangle f_k\Bigr),\\
\eta_{kj}^*(z,x,f) = &\;\alpha_{ij}^*\gamma_{pq}\beta_{kl}^*(z,x,f) = \Bigl(\langle x,f_k \rangle \langle c_{pq}, v_{kl} \rangle \langle f_p,x_i \rangle w_{ij},0,0 \Bigr),\\
\frac{1}{2}\eta_{kj}\eta_{kj}^*(z,x,f) = &\Bigl(0,0,\langle x,f_k \rangle
\langle c_{kl}, w_{pq} \rangle^2 \langle f_p,x_i \rangle^2 q(w_{ij})
f_k \Bigr),\\
 E_{\eta_{kj}}^*(z,x,f) = &\;\Bigl(
                    z - \langle x,f_k \rangle \langle c_{pq}, v_{kl} \rangle \langle x_i,f_p \rangle w_{ij}, x ,\; f+ \langle w_{ij},z \rangle \langle c_{pq},v_{kl}\rangle \langle x_i,f_p \rangle f_k
                  \Bigr.\\
              &\;\Bigl.
                    - \langle x,f_k \rangle \langle c_{pq}, v_{kl} \rangle^2 \langle x_i,f_p \rangle^2 q(w_{ij}) f_k
                  \Bigr).
\end{align*}
If $i \neq k$, then, by Lemma~\ref{l02}, we have
\begin{align*}
  [E_{\alpha_{ij}},E_\frac{\eta_{kj}}{2}^*](z,x,f)
  =\;&\Bigl( I +\frac{1}{2}\eta_{kj}\alpha_{ij}^* - \frac{1}{2}\alpha_{ij}\eta_{kj}^*\Bigr)(z,x,f)\\
  =\;&\Bigl(z,\; x- \langle x,f_k \rangle \langle c_{pq}, v_{kl} \rangle\langle x_i,f_p \rangle q(w_{ij}) x_i,\Bigr.\\
  &\hspace{1cm}\Bigl.f+ \langle f,x_i \rangle \langle c_{pq},v_{kl}\rangle \langle x_i,f_p \rangle q(w_{ij})f_k\Bigr)
\end{align*}
and hence we get
\begin{align*}
 E_{\eta_{kj}}^*&[E_{\alpha_{ij}},E_\frac{\eta_{kj}}{2}^*] (z,x,f)\\
      &= \Bigl( I +\eta_{kj}-\eta_{kj}^*-\frac{1}{2}\eta_{kj}\eta_{kj}^*+\frac{1}{2}\eta_{kj}\alpha_{ij}^* - \frac{1}{2}\alpha_{ij}\eta_{kj}^*\Bigr)(z,x,f)\\
        &= \Bigl(z - \langle x,f_k \rangle \langle c_{pq},v_{kl} \rangle \langle x_i,f_p \rangle w_{ij},\;x-\langle x,f_k \rangle \langle c_{pq},v_{kl}\rangle\langle x_i,f_p \rangle q(w_{ij})x_i,\; f +\Bigl\{
          \langle  w_{ij},z \rangle\Bigr.\Bigr.\\
          &\hspace{1cm}  
            +\langle f,x_i \rangle  q(w_{ij}) -\langle x,f_k \rangle \Bigl. \langle c_{pq},v_{kl} \rangle \langle x_i,f_p \rangle q(w_{ij})
          \Bigr\}
          \langle c_{pq},v_{kl} \rangle \langle x_i,f_p \rangle f_k
          \Bigr).
          \numberthis{eqnt13}
\end{align*}
Similarly, if $i \neq p$, then we have
      \begin{align*}
          E_{\vartheta_{pj}}^*&\left[E_{\alpha_{ij}},E_\frac{\vartheta_{pj}}{2}^* \right] (z,x,f)\\
           &= \Bigl ( I + \vartheta_{pj} - \vartheta_{pj}^* - \frac{1}{2}\vartheta_{pj}\vartheta_{pj}^* - \frac{1}{2} \alpha_{ij} \vartheta_{pj}^* + \frac{1}{2} \vartheta_{pj} \alpha_{ij}^*   \Bigr)(z,x,f)\\
           &= \Bigl( z +\langle x,f_p \rangle \langle c_{pq},v_{kl}  \rangle \langle x_i, f_k \rangle w_{ij}, \; x - \langle x,f_p \rangle \langle c_{pq},v_{kl} \rangle\langle x_i, f_k \rangle q(w_{ij})x_i,\; f  - \Bigl\{\langle  w_{ij},z \rangle \\
      & \hspace{1cm}  - \langle f,x_i \rangle  q(w_{ij}) -\langle x,f_p \rangle   \langle c_{pq},v_{kl} \rangle \langle x_i, f_k \rangle  q(w_{ij})\Bigr\} \langle x_i, f_k \rangle\langle c_{pq},v_{kl} \rangle f_p\Bigr).
       \numberthis{eqnt14}
      \end{align*}
We now consider the following possible conditions on the indices.

     \smallskip

\noindent\textbf{Case(i):} $i = p$.

  \smallskip

      \noindent If $i = p$, then, by Equations~\eqref{eqnt11},~\eqref{eqnt12}, and~\eqref{eqnt13}, we have
  \begin{align*}
      \left[E_{\alpha_{ij}},\right.&\left.\left[ E_{\beta_{kl}}^*,E_{\gamma_{pq}}^* \right]\right](z,x,f)\\
      &=\Bigl(I -\alpha_{pj}^*\gamma_{pq}\beta_{kl}^*
                    - \frac{1}{2}\alpha_{pj}\alpha_{pj}^*\gamma_{pq}\beta_{kl}^*
                    +\beta_{kl}\gamma_{pq}^*\alpha_{pj}+\frac{1}{2}\beta_{kl}\gamma_{pq}^*\alpha_{pj}\alpha_{pj}^*
        \Bigr.\\
              &\hspace{1cm}
              \Bigl.
              - \frac{1}{2}\beta_{kl}\gamma_{pq}^*\alpha_{pj}\alpha_{pj}^*\gamma_{pq}\beta_{kl}^*
              \Bigr)(z,x,f)\\
              &= \Bigl(z - \langle x,f_k \rangle \langle c_{pq},v_{kl} \rangle w_{pj},\;x-\langle x,f_k \rangle \langle c_{pq},v_{kl} \rangle q(w_{pj})x_p, \Bigr. \\
            & \hspace{1cm} \Bigl. f + \Bigl\{\langle  w_{pj},z \rangle   
            +\langle f,x_p \rangle  q(w_{pj}) 
            -\langle x,f_k \rangle \langle c_{pq},v_{kl} \rangle  q(w_{pj})\Bigr\}\langle c_{pq},v_{kl} \rangle f_k\Bigr)\\
            &= E_{\eta_{kj}}^* \left[ E_{\alpha_{ij}},E_\frac{\eta_{kj}}{2}^* \right] (z,x,f).
   \end{align*}
   \noindent \textbf{Case(ii):} $i = k$.

    \smallskip

    \noindent If $i = k$, then, by Equations~\eqref{eqnt11},~\eqref{eqnt12}, and~\eqref{eqnt14}, we have
    \begin{align*}
    \left[E_{\alpha_{ij}},\right.&\left[E_{\beta_{kl}}^*,E_{\gamma_{pq}}^*\left.\right]\right](z,x,f)\\
    &= \Bigl(I -\gamma_{pq}\beta_{kl}^*\alpha_{kj}
                + \alpha_{kj}^*\beta_{kl}\gamma_{pq}^* + \frac{1}{2}\alpha_{kj}\alpha_{kj}^*\beta_{kl}\gamma_{pq}^*- \frac{1}{2}\gamma_{pq}\beta_{kl}^*\alpha_{kj}\alpha_{kj}^*
          \Bigr.\\
              &\hspace{1cm} \Bigl.
                          - \frac{1}{2}\gamma_{pq}\beta_{kl}^*\alpha_{kj}\alpha_{kj}^*\beta_{kl}\gamma_{pq}^*
                   \Bigr)(z,x,f)\\
    =\;& \Bigl( z +\langle x,f_p \rangle \langle c_{pq},v_{kl}  \rangle w_{kj}, \; x+  \langle x,f_p \rangle\langle c_{pq},v_{kl} \rangle q(w_{kj})x_k,\\
      & \hspace{1cm} f  - \Bigl\{\langle  w_{kj},z \rangle  + \langle f,x_k \rangle q(w_{kj}) +
      \langle x,f_p \rangle \langle c_{pq},v_{kl} \rangle  q(w_{kj})\Bigr\}  \langle c_{pq},v_{kl} \rangle f_p\Bigr).
    \end{align*}
    \noindent \textbf{Case(iii):} $i \neq p$ and $i \neq k$.

    \smallskip

    \noindent If $i \neq p$, then, by Equation~\eqref{eqnt11}, we have
    \[
    \left[E_{\alpha_{ij}},\left[E_{\beta_{kl}}^*,E_{\gamma_{pq}}^*\right]\right](z,x,f)
    = I(z,x,f).\qedhere
    \]
\end{proof}
\begin{cor}\label{c07}
 For any $i,j,k,l,p,q$ with $1 \leq i,k,p \leq m$, $1 \leq j,l,q \leq n$,
$i \neq k$ and $k \neq p$ and $a,b,c,d,e,f \in A$ with $abc=def$ and $a^2bc=d^2ef$, the following equation holds.
\[ \left[E_{a\alpha_{ij}},\left[E_{b\beta_{kl}}^*,E_{c\gamma_{pq}}^*\right]\right]=\left[E_{d\alpha_{ij}},\left[E_{e\beta_{kl}}^*,E_{f\gamma_{pq}}^*\right]\right]. \]
\end{cor}
\section{Multiple Commutators}
 In this section, we establish some four-fold commutator formulae using coordinate-free method.
 
 \begin{lem}\label{l08}
  Let $\alpha \in \hom_A(Q,P)$ and $\beta,\gamma, \mu \in \hom_A(Q,P^*)$. Then, for $i,j,k,l,r,s,p,q$ with $1 \leq i,k,r,p \leq m$, $1 \leq j,l,s,q \leq n$, $i \neq k$ and $r \neq p$, the four-fold commutator
  $\Bigl[[ E_{\beta_{ij}}^*,E_{\gamma_{kl}}^* ], [ E_{\alpha_{rs}}, E_{\mu_{pq}}^* ]\Bigr]$ is given by
  \[
    \left[[E_{\beta_{ij}}^*,E_{\gamma_{kl}}^*], [E_{\alpha_{rs}}, E_{\mu_{pq}}^*]\right]
    =   \begin{cases}
    \vspace{2mm}
    \left[ E_{\mu_{pq}\alpha_{rs}^*}^*, E_{\beta_{ij}\gamma_{kl}^*}^*   \right]
    &\textnormal{if}\quad  k=r,\\
      \vspace{2mm}
    \left[ E_{\gamma_{kl}\beta_{ij}^*}^*,E_{\mu_{pq}\alpha_{rs}^*}^* \right]
    &\textnormal{if}\quad  i = r,\\
    \vspace{2mm}
     I & \textnormal{ otherwise }. \\
           \end{cases}
  \]
 \end{lem}
 \begin{proof}
  If $i \neq k$, then, by Lemma~\ref{l03}, we have
  \[
  [E_{\beta_{ij}}^*,E_{\gamma_{kl}}^*](z,x,f) = (I + \gamma_{kl}\beta_{ij}^* -
\beta_{ij}\gamma_{kl}^* )(z,x,f).
 \]
 If $r \neq p$, then, by Lemma~\ref{l02}, we have
 \[
  [E_{\alpha_{rs}}, E_{\mu_{pq}}^* ](z,x,f) = (I + \mu_{pq} \alpha_{rs}^* -
\alpha_{rs}\mu_{pq}^* )(z,x,f).
 \]
Now if $i \neq k$ and $r \neq p$, then we get
\begin{align*}
 \left[[E_{\beta_{ij}}^*,E_{\gamma_{kl}}^*],\right.&\left. [E_{\alpha_{rs}}, E_{\mu_{pq}}^*
]\right] (z,x,f)\\
 &=  (I - \gamma_{kl}\beta_{ij}^*\alpha_{rs}\mu_{pq}^* - \mu_{pq}\alpha_{rs}^*\gamma_{kl}\beta_{ij}^* + \beta_{ij}\gamma_{kl}^*\alpha_{rs}\mu_{pq}^*   + \mu_{pq}\alpha_{rs}^* \beta_{ij}\gamma_{kl}^*)(z,x,f)\\
& =  \left[E_{\mu_{pq}\alpha_{rs}^*}^*, E_{\beta_{ij}\gamma_{kl}^*}^* \right]\left[ E_{\gamma_{kl}\beta_{ij}^*}^*,E_{\mu_{pq}\alpha_{rs}^*}^* \right](z,x,f).
\numberthis{eq1}
\end{align*}
Now if $k =r$, then Equation \eqref{eq1} becomes $\left[ E_{\mu_{pq}\alpha_{rs}^*}^*, E_{\beta_{ij}\gamma_{kl}^*}^* \right] (z,x,f)$ and in particular
 \[     {\left[ E_{\mu_{pq}\alpha_{rs}^*}^*, E_{\beta_{ij}\gamma_{kl}^*}^* \right]} (z,x,f)=
      I (z,x,f) \quad\textnormal{ if }\quad  i = p,
  \]
and if $i = r$, then Equation \eqref{eq1} becomes
$\left[ E_{\gamma_{kl}\beta_{ij}^*}^*,E_{\mu_{pq}\alpha_{rs}^*}^* \right](z,x,f)$ and in particular
\[
    \left[ E_{\gamma_{kl}\beta_{ij}^*}^*,E_{\mu_{pq}\alpha_{rs}^*}^* \right](z,x,f) =
           I (z,x,f) \quad\textnormal{ if }\quad  k = p.\qedhere
   \]
\end{proof}
\begin{lem}\label{l09}
  Let $\alpha, \delta, \xi \in \hom_A(Q,P)$ and $\beta \in \hom_A(Q,P^*)$. Then, for $i,j,k,l,r,s,p,q$ with $1 \leq i,k,r,p \leq m$, $1 \leq j,l,s,q \leq n$, $i \neq k$ and $s \neq p$, the four-fold commutator $\left[[E_{\alpha_{ij}},E_{\delta_{kl}}],[E_{\xi_{rs}}, E_{\beta_{pq}}^* ]\right]$ is given by
  \[
  \left[[E_{\alpha_{ij}},E_{\delta_{kl}}], [E_{\xi_{rs}}, E_{\beta_{pq}}^* ]\right]
    = \begin{cases}
    \vspace{2mm}
    \left[  E_{\delta_{kl}\alpha_{ij}^*}, E_{\xi_{rs}\beta_{pq}^*} \right]
    &\textnormal{ if }\quad i = p,\\
        \vspace{2mm}
    \left[ E_{\alpha_{ij}\delta_{kl}^*}, E_{\xi_{rs}\beta_{pq}^*} \right]
    &\textnormal{ if }\quad k = p,\\
        \vspace{2mm}
     I & \textnormal{ otherwise }.
      \end{cases}
 \]
 \end{lem}
\begin{proof}
  If $i \neq k$, then, by Lemma~\ref{l01}, we have
 \[
  [E_{\alpha_{ij}},E_{\delta_{kl}}](z,x,f) = (I + \delta_{kl}\alpha_{ij}^* -
\alpha_{ij}\delta_{kl}^* )(z,x,f).
 \]
 If $r \neq p$, then, by Lemma~\ref{l02}, we have
 \[
  [E_{\xi_{rs}}, E_{\beta_{pq}}^* ](z,x,f) = (I + \beta_{pq} \xi_{rs}^* -
\xi_{rs}\beta_{pq}^* )(z,x,f).
 \]
Now if $i \neq k$ and $r \neq p$, then we get
\begin{align*}
  \left[[E_{\alpha_{ij}},E_{\delta_{kl}}],\right.&\left. [E_{\xi_{rs}}, E_{\beta_{pq}}^*
]\right](z,x,f)\\
  &=(I + \delta_{kl} \alpha_{ij}^* \beta_{pq}\xi{rs}^*
 - \alpha_{ij} \delta_{kl}^* \beta_{pq}\xi_{rs}^*
 + \xi_{rs}\beta_{pq}^* \delta_{kl} \alpha_{ij}^*   \\
 &\hspace{1cm}- \xi_{rs}\beta_{pq}^*\alpha_{ij}\delta_{kl}^*- \xi_{rs}\beta_{pq}^* \delta_{kl} \alpha_{ij}^*\beta_{pq}\xi_{rs}^* + \xi_{rs}\beta_{pq}^*\alpha_{ij}\delta_{kl}^* \beta_{pq}\xi_{rs}^* )(z,x,f) \\
  & = \left[ E_{\alpha_{ij}\delta_{kl}^*}, E_{\xi_{rs}\beta_{pq}^*}\right] \left[ E_{\delta_{kl}\alpha_{ij}^*}, E_{\xi_{rs}\beta_{pq}^*}\right](z,x,f).
 \numberthis{eq2}
 \end{align*}
 Now if $k = p$, then Equation \eqref{eq2} becomes $\left[ E_{\alpha_{ij}\delta_{kl}^*}, E_{\xi_{rs}\beta_{pq}^*}\right](z,x,f)$ and in particular
 \[
      \left[ E_{\alpha_{ij}\delta_{kl}^*}, E_{\xi_{rs}\beta_{pq}^*}\right](z,x,f) = I(z,x,f)\quad\textnormal{ if }\quad i=r,\\
    \]
and if $i = p$, then Equation \eqref{eq2} becomes $ \left[ E_{\delta_{kl}\alpha_{ij}^*}, E_{\xi_{rs}\beta_{pq}^*}\right](z,x,f)$ and in particular
\[
    \left[ E_{\delta_{kl}\alpha_{ij}^*}, E_{\xi_{rs}\beta_{pq}^*}\right](z,x,f)
    = I(z,x,f) \quad \textnormal{ if } \quad k=r.\qedhere
  \]
\end{proof}
\begin{lem}\label{l10}
  Let $\alpha, \delta \in \hom_A(Q,P)$ and $\beta,\gamma \in \hom_A(Q,P^*)$. Then,
for any $i,j,k,l,r,$ $s,p,q$ with $1 \leq i,k,r,p \leq m$, $1 \leq j,l,s,q \leq n$, $i
\neq k$ and $r \neq p$, the four-fold commutator $\left[ [E_{\alpha_{ij}},
E_{\beta_{kl}}^*], [ E_{\delta_{rs}}, E_{\gamma_{pq}}^*] \right]$is given by
\[
  \left[ [E_{\alpha_{ij}}, E_{\beta_{kl}}^*], [ E_{\delta_{rs}},E_{\gamma_{pq}}^*] \right]
= \begin{cases}
   \vspace{2mm}
   {\left[ E_{\alpha_{ij}\beta_{kl}^*},E_{\gamma_{pq}\delta_{rs}^*}^*\right]}^{-1} & \textnormal{ if }\quad  k=r \textnormal{ and } i \neq p, \\
     \vspace{2mm}
    \left[ E_{\delta_{rs}\gamma_{pq}^*}, E_{\beta_{kl}\alpha_{ij}^*}^* \right] & \textnormal{ if }\quad i = p \textnormal{ and } k \neq r,\\
    I & \textnormal{ if }\quad k \neq r \textnormal{ and } i \neq p.
  \end{cases}
 \]
 \end{lem}
\begin{proof}
 If $i \neq k$, then, by Lemma~\ref{l02}, we have
 \[
  [E_{\alpha_{ij}}, E_{\beta_{kl}}^*](z,x,f) = (I + \beta_{kl}\alpha_{ij}^* -
\alpha_{ij}\beta_{kl}^* )(z,x,f).
 \]
 If $r \neq p$, then, by Lemma~\ref{l02}, we have
 \[
  [ E_{\delta_{rs}}, E_{\gamma_{pq}}^*](z,x,f) = (I  + \gamma_{pq}\delta_{rs}^*
- \delta_{rs} \gamma_{pq}^* )(z,x,f).
 \]
Now if $i \neq k$ and $r \neq p$, then we get
\begin{align*}
 \left[ [E_{\alpha_{ij}}, E_{\beta_{kl}}^*], \right.&\left.[ E_{\delta_{rs}},E_{\gamma_{pq}}^*] \right](z,x,f)\\
       & =(I + \beta_{kl}\alpha_{ij}^* \gamma_{pq}\delta_{rs}^*                           +\alpha_{ij}\beta_{kl}^*\delta_{rs}\gamma_{pq}^*
     -\gamma_{pq}\delta_{rs}^*\beta_{kl}\alpha_{ij}^* + \alpha_{ij}\beta_{kl}^* \delta_{rs}\gamma_{pq}^*\alpha_{ij}\beta_{kl}^*
     \\
 &\hspace{1cm}+ \alpha_{ij}\beta_{kl}^* \delta_{rs}\gamma_{pq}^*\alpha_{ij}\beta_{kl}^*\delta_{rs}\gamma_{pq}^*
   -\delta_{rs} \gamma_{pq}^*\alpha_{ij}\beta_{kl}^* \delta_{rs} \gamma_{pq}^*
  -\beta_{kl}\alpha_{ij}^*\gamma_{pq}\delta_{rs}^*\beta_{kl}\alpha_{ij}^*
     \\
  &\hspace{1cm} -\delta_{rs} \gamma_{pq}^*\alpha_{ij}\beta_{kl}^*+ \gamma_{pq}\delta_{rs}^* \beta_{kl}\alpha_{ij}^*  \gamma_{pq}\delta_{rs}^* +\beta_{kl}\alpha_{ij}^*\gamma_{pq}\delta_{rs}^*\beta_{kl}\alpha_{ij}^*\gamma_{pq}\delta_{rs}^* )(z,x,f).
   \numberthis{eq3}
\end{align*}
Now if $k = r$ and $i \neq p$, then, by Equation~\eqref{eq3}, we have
 \[
     \left[ [E_{\alpha_{ij}}, E_{\beta_{kl}}^*], [ E_{\delta_{rs}},E_{\gamma_{pq}}^*] \right](z,x,f) = {\left[ E_{\alpha_{ij}\beta_{kl}^*},E_{\gamma_{pq}\delta_{rs}^*}^*\right]}^{-1} (z,x,f)
      \]
and if $i = p$ and $ k \neq r$, then, by Equation~\eqref{eq3}, we have
\[
     \left[ [E_{\alpha_{ij}}, E_{\beta_{kl}}^*], [ E_{\delta_{rs}},E_{\gamma_{pq}}^*] \right](z,x,f) =
     \left[  E_{\delta_{rs}\gamma_{pq}^*}, E_{\beta_{kl}\alpha_{ij}^*}^* \right](z,x,f).
\]
Now if $i \neq p$ and $k \neq r$, then, by Equation~\eqref{eq3}, we get
\[
     \left[ [E_{\alpha_{ij}}, E_{\beta_{kl}}^*], [ E_{\delta_{rs}},E_{\gamma_{pq}}^*] \right](z,x,f) =
     I(z,x,f).\qedhere
\]
\end{proof}
\begin{lem}\label{l11}
  Let $\alpha, \delta, \xi, \mu \in \hom_A(Q,P)$. Then, for $i,j,k,l,r,s,p,q$ with ${1 \leq i,k,r,p \leq m}$, $1 \leq j,l,s,q \leq n$, $i \neq k$ and $r \neq p$, the four-fold commutator $\left[ [E_{\alpha_{ij}},E_{\delta_{kl}}], [ E_{\xi_{rs}}, E_{\mu_{pq}}] \right]$ is given by
\[
  \left[ [E_{\alpha_{ij}},E_{\delta_{kl}}], [ E_{\xi_{rs}}, E_{\mu_{pq}}] \right]
=    I .
 \]
 \end{lem}
\begin{proof}
 If $i \neq k$, then, by Lemma~\ref{l01}, we have
 \[
  [E_{\alpha_{ij}}, E_{\delta_{kl}}](z,x,f) = (I + \delta_{kl}\alpha_{ij}^* - \alpha_{ij}\delta_{kl}^*)(z,x,f).
 \]
 If $r \neq p$, then, by Lemma~\ref{l01}, we have
 \[
 [ E_{\xi_{rs}}, E_{\mu_{pq}}](z,x,f) = (I + \mu_{pq}\xi_{rs}^* - \xi_{rs}\mu_{pq}^*)  (z,x,f).
 \]
Now if $i \neq k$ and $r \neq p$, then we get
\begin{align*}
 \left[ [E_{\alpha_{ij}},E_{\delta_{kl}}],\right.&\left. [ E_{\xi_{rs}}, E_{\mu_{pq}}] \right](z,x,f)\\
 & =  [E_{\alpha_{ij}},E_{\delta_{kl}}] [ E_{\xi_{rs}}, E_{\mu_{pq}}]
      {[E_{\alpha_{ij}},E_{\delta_{kl}}]}^{-1} {[ E_{\xi_{rs}}, E_{\mu_{pq}}]}^{-1}(z,x,f)\\
 & =  [E_{\alpha_{ij}},E_{\delta_{kl}}] [ E_{\xi_{rs}}, E_{\mu_{pq}}]
      {[E_{\alpha_{ij}},E_{\delta_{kl}}]}^{-1}\Bigl(\Bigl (I - \mu_{pq}\xi_{rs}^*  
   +\, \xi_{rs}\mu_{pq}^* \Bigr)(z,x,f) \Bigr)\\
 & = [E_{\alpha_{ij}},E_{\delta_{kl}}] [ E_{\xi_{rs}}, E_{\mu_{pq}}] \Bigl(\Bigl (I - \mu_{pq}\xi_{rs}^*  + \xi_{rs}\mu_{pq}^*  + \delta_{kl}\alpha_{ij}^*  - \alpha_{ij}\delta_{kl}^* \Bigr)(z,x,f) \Bigr)\\
 & = [E_{\alpha_{ij}},E_{\delta_{kl}}]\Bigl(\Bigl (I  + \delta_{kl}\alpha_{ij}^*  - \alpha_{ij}\delta_{kl}^*\Bigr)(z,x,f)\Bigr)\\
 & = I(z,x,f).\qedhere
 \end{align*}
\end{proof}
\begin{lem}\label{l12}
  Let $\beta, \gamma, \eta, \nu \in \hom_A(Q,P)$. Then, for $i,j,k,l,r,s,p,q$ with ${1 \leq i,k,r,p \leq m}$, $1 \leq j,l,s,q \leq n$, $i \neq k$ and $r \neq p$, the four-fold commutator $\left[ [E_{\beta_{ij}}^*,E_{\gamma_{kl}}^*], [E_{\eta_{rs}}^*, E_{\nu_{pq}}^*] \right]$ is given by
\[
  \left[ [E_{\beta_{ij}}^*,E_{\gamma_{kl}}^*], [ E_{\eta_{rs}}^*, E_{\nu_{pq}}^*] \right]
=    I.
 \]
 \end{lem}
\begin{proof}
 If $i \neq k$, then, by Lemma~\ref{l03}, we have
 \[
  [E_{\beta_{ij}}^*,E_{\gamma_{kl}}^*](z,x,f) = (I + \gamma_{kl}\beta_{ij}^* - \beta_{ij}\gamma_{kl}^*)(z,x,f).
 \]
 If $r \neq p$, then, by Lemma~\ref{l03}, we have
 \[
 [  E_{\eta_{rs}}^*, E_{\nu_{pq}}^* ](z,x,f) = (I + \nu_{pq}\eta_{rs}^* - \eta_{rs}\nu_{pq}^*)  (z,x,f).
 \]
Now if $i \neq k$ and $r \neq p$, then we get
\begin{align*}
 \left[[E_{\beta_{ij}}^*,E_{\gamma_{kl}}^*],\right.&\left. [ E_{\eta_{rs}}^*, E_{\nu_{pq}}^*]] \right](z,x,f)\\
 & =  [E_{\beta_{ij}}^*,E_{\gamma_{kl}}^*][ E_{\eta_{rs}}^*, E_{\nu_{pq}}^*]
      {[E_{\beta_{ij}}^*,E_{\gamma_{kl}}^*] }^{-1} {[ E_{\eta_{rs}}^*, E_{\nu_{pq}}^*]}^{-1}(z,x,f)\\
 & =  [E_{\beta_{ij}}^*,E_{\gamma_{kl}}^*][ E_{\eta_{rs}}^*, E_{\nu_{pq}}^*]
      {[E_{\beta_{ij}}^*,E_{\gamma_{kl}}^*] }^{-1}\Bigl(\Bigl (I - \nu_{pq}\eta_{rs}^*  
 + \eta_{rs}\nu_{pq}^* \Bigr)(z,x,f) \Bigr)\\
 & = [E_{\beta_{ij}}^*,E_{\gamma_{kl}}^*][ E_{\eta_{rs}}^*, E_{\nu_{pq}}^*]
      \Bigl(\Bigl (
            I  -  \nu_{pq}\eta_{rs}^* - \eta_{rs}\nu_{pq}^* - \gamma_{kl}\beta_{ij}^* 
   + \beta_{ij}\gamma_{kl}^* \Bigr)(z,x,f) \Bigr)\\
 & =  [E_{\beta_{ij}}^*,E_{\gamma_{kl}}^*]\Bigl( \Bigl(I - \gamma_{kl}\beta_{ij}^* + \beta_{ij}\gamma_{kl}^* \Bigr)(z,x,f)\Bigr)\\
 & = I(z,x,f).\qedhere
 \end{align*}
\end{proof}
\begin{lem}\label{l13}
  Let $\alpha, \delta \in \hom_A(Q,P)$ and $\beta,\gamma \in \hom_A(Q,P^*)$. Then, 
for any $i,j,k,l,r,$ $s,p,q$ with $1 \leq i,k,r,p \leq m$, $1 \leq j,l,s,q \leq n$, $i
\neq k$ and $r \neq p$, the four-fold commutator $\left[ [E_{\alpha_{ij}}, E_{\delta_{kl}}], [ E_{\beta_{rs}}^*,E_{\gamma_{pq}}^*] \right]$ is given by
\[
  \left[ [E_{\alpha_{ij}}, E_{\delta_{kl}}], [ E_{\beta_{rs}}^*,E_{\gamma_{pq}}^*] \right]
= \begin{cases}
   \vspace{2mm}
   {\left[ E_{\alpha_{ij}\beta_{kl}^*},E_{\gamma_{pq}\delta_{rs}^*}^*\right]}^{-1} & \textnormal{ if }\quad  k=r \textnormal{ and } i \neq p, \\
     \vspace{2mm}
    \left[ E_{\delta_{rs}\gamma_{pq}^*}, E_{\beta_{kl}\alpha_{ij}^*}^* \right] & \textnormal{ if }\quad i = p \textnormal{ and } k \neq r,\\
    I & \textnormal{ if }\quad k \neq r \textnormal{ and } i \neq p.
  \end{cases}
 \]
 \end{lem}
\begin{proof}
 If $i \neq k$, then, by Lemma~\ref{l01}, we have
 \[
  [E_{\alpha_{ij}}, E_{\delta_{kl}}](z,x,f) = (I + \delta_{kl}\alpha_{ij}^* - \alpha_{ij}\delta_{kl}^*)(z,x,f).
 \]
 If $r \neq p$, then, by Lemma~\ref{l03}, we have
 \[
  [E_{\beta_{rs}}^*,E_{\gamma_{pq}}^*](z,x,f) = (I + \gamma_{pq}\beta_{rs}^* - \beta_{rs}\gamma_{pq}^*)(z,x,f).
 \]
Now if $i \neq k$ and $r \neq p$, then, by the coordinate-free method, we get
\begin{align*}
\left[ [E_{\alpha_{ij}}, E_{\delta_{kl}}],\right.&\left. [ E_{\beta_{rs}}^*,E_{\gamma_{pq}}^*] \right](z,x,f)\\
     &=\Bigl(I + \delta_{kl}\alpha_{ij}^*\gamma_{pq}\beta_{rs}^*
		-\delta_{kl}\alpha_{ij}^*\beta_{rs}\gamma_{pq}^*
		-\alpha_{ij}\delta_{kl}^*\gamma_{pq}\beta_{rs}^* 
		+ \alpha_{ij}\delta_{kl}^* \beta_{rs}\gamma_{pq}^*
           - \gamma_{pq}\beta_{rs}^*\delta_{kl}\alpha_{ij}^*\Bigr.\\
      &\hspace{1cm}  + \gamma_{pq}\beta_{rs}^* \alpha_{ij}\delta_{kl}^*
          -\beta_{rs}\gamma_{pq}^*\alpha_{ij}\delta_{kl}^*
          -\beta_{rs}\gamma_{pq}^*\alpha_{ij}\delta_{kl}^*\beta_{rs}\gamma_{pq}^*+\alpha_{ij}\delta_{kl}^* \gamma_{pq}\beta_{rs}^*\delta_{kl}\alpha_{ij}^*
          \\
&\hspace{1cm} +\beta_{rs}\gamma_{pq}^*\delta_{kl}\alpha_{ij}^*
+\gamma_{pq}\beta_{rs}^*\delta_{kl}\alpha_{ij}^*\gamma_{pq}\beta_{rs}^*
- \gamma_{pq}\beta_{rs}^*\delta_{kl}\alpha_{ij}^*\beta_{rs}\gamma_{pq}^*
        \\
&\hspace{1cm}+ \gamma_{pq}\beta_{rs}^* \alpha_{ij}\delta_{kl}^*\beta_{rs}\gamma_{pq}^*
-\beta_{rs}\gamma_{pq}^*\delta_{kl}\alpha_{ij}^*\gamma_{pq}\beta_{rs}^*
+\beta_{rs}\gamma_{pq}^*\delta_{kl}\alpha_{ij}^* \beta_{rs}\gamma_{pq}^*
\\
&\hspace{1cm}+\beta_{rs}\gamma_{pq}^*\alpha_{ij}\delta_{kl}^*\gamma_{pq}\beta_{rs}^*
- \delta_{kl}\alpha_{ij}^*\gamma_{pq}\beta_{rs}^*\delta_{kl}\alpha_{ij}^*
+\delta_{kl}\alpha_{ij}^* \gamma_{pq}\beta_{rs}^* \alpha_{ij}\delta_{kl}^*\\
&\hspace{1cm}+\delta_{kl}\alpha_{ij}^*\beta_{rs}\gamma_{pq}^* \delta_{kl}\alpha_{ij}^*
            -\delta_{kl}\alpha_{ij}^*\beta_{rs}\gamma_{pq}^*\alpha_{ij}\delta_{kl}^*
             +\alpha_{ij}\delta_{kl}^*\beta_{rs}\gamma_{pq}^*\alpha_{ij}\delta_{kl}^*
          \\
          &\hspace{1cm} -\alpha_{ij}\delta_{kl}^* \gamma_{pq}\beta_{rs}^* \alpha_{ij}\delta_{kl}^*
           -\alpha_{ij}\delta_{kl}^*\beta_{rs}\gamma_{pq}^* \delta_{kl}\alpha_{ij}^*- \gamma_{pq}\beta_{rs}^* \alpha_{ij}\delta_{kl}^*\gamma_{pq}\beta_{rs}^*
        \\
           &\hspace{1cm}-\delta_{kl}\alpha_{ij}^*\beta_{rs}\gamma_{pq}^*\alpha_{ij}\delta_{kl}^*\beta_{rs}\gamma_{pq}^*
          +\alpha_{ij}\delta_{kl}^*\beta_{rs}\gamma_{pq}^*\alpha_{ij}\delta_{kl}^*\beta_{rs}\gamma_{pq}^*
          \\
        &\hspace{1cm} + \delta_{kl}\alpha_{ij}^*\gamma_{pq}\beta_{rs}^*\delta_{kl}\alpha_{ij}^*\gamma_{pq}\beta_{rs}^*- \delta_{kl}\alpha_{ij}^*\gamma_{pq}\beta_{rs}^*\delta_{kl}\alpha_{ij}^*\beta_{rs}\gamma_{pq}^*
        \\
        &\hspace{1cm} - \delta_{kl}\alpha_{ij}^*\gamma_{pq}\beta_{rs}^* \alpha_{ij}\delta_{kl}^*\gamma_{pq}\beta_{rs}^*+ \delta_{kl}\alpha_{ij}^*\gamma_{pq}\beta_{rs}^* \alpha_{ij}\delta_{kl}^*\beta_{rs}\gamma_{pq}^*
        \\
        &\hspace{1cm} -\delta_{kl}\alpha_{ij}^*\beta_{rs}\gamma_{pq}^*\delta_{kl}\alpha_{ij}^*\gamma_{pq}\beta_{rs}^*
        +\delta_{kl}\alpha_{ij}^*\beta_{rs}\gamma_{pq}^*\delta_{kl}\alpha_{ij}^* \beta_{rs}\gamma_{pq}^*
        \\
        &\hspace{1cm} -\alpha_{ij}\delta_{kl}^* \gamma_{pq}\beta_{rs}^*\delta_{kl}\alpha_{ij}^*\gamma_{pq}\beta_{rs}^*
        +\alpha_{ij}\delta_{kl}^*\gamma_{pq}\beta_{rs}^*\delta_{kl}\alpha_{ij}^*\beta_{rs}\gamma_{pq}^*
        \\
        &\hspace{1cm}+\alpha_{ij}\delta_{kl}^* \gamma_{pq}\beta_{rs}^* \alpha_{ij}\delta_{kl}^*\gamma_{pq}\beta_{rs}^*
        -\alpha_{ij}\delta_{kl}^*\gamma_{pq}\beta_{rs}^* \alpha_{ij}\delta_{kl}^*\beta_{rs}\gamma_{pq}^*
        \\
        &\hspace{1cm} +\alpha_{ij}\delta_{kl}^*\beta_{rs}\gamma_{pq}^*\delta_{kl}\alpha_{ij}^*\gamma_{pq}\beta_{rs}^*
        -\alpha_{ij}\delta_{kl}^*\beta_{rs}\gamma_{pq}^*\delta_{kl}\alpha_{ij}^* \beta_{rs}\gamma_{pq}^*
        \\
        &\hspace{1cm}
        +\delta_{kl}\alpha_{ij}^*\beta_{rs}\gamma_{pq}^*\alpha_{ij}\delta_{kl}^*\gamma_{pq}\beta_{rs}^*-\alpha_{ij}\delta_{kl}^*\beta_{rs}\gamma_{pq}^*\alpha_{ij}\delta_{kl}^*\gamma_{pq}\beta_{rs}^*)(z,x,f).
          \numberthis{eq4}
 \end{align*}
If $i = p$ or $k=r$ or $i \neq r$ and $k \neq p$, then the Equation~\eqref{eq4} becomes
\begin{align*}
 \left[ [E_{\alpha_{ij}}, E_{\delta_{kl}}],\right.&\left. [ E_{\beta_{rs}}^*,E_{\gamma_{pq}}^*] \right](z,x,f)\\
& =I +\delta_{kl}\alpha_{ij}^*\gamma_{pq}\beta_{rs}^* + \alpha_{ij}\delta_{kl}^*\beta_{rs}\gamma_{pq}^* - \gamma_{pq}\beta_{rs}^*\delta_{kl}\alpha_{ij}^*-\beta_{rs}\gamma_{pq}^*\alpha_{ij}\delta_{kl}^* \\
&\hspace{1cm} - \beta_{rs}\gamma_{pq}^*\alpha_{ij}\delta_{kl}^*\beta_{rs}\gamma_{pq}^*
+\gamma_{pq}\beta_{rs}^*\delta_{kl}\alpha_{ij}^*\gamma_{pq}\beta_{rs}^*- \delta_{kl}\alpha_{ij}^*\gamma_{pq}\beta_{rs}^*\delta_{kl}\alpha_{ij}^*\\
&\hspace{1cm}+\alpha_{ij}\delta_{kl}^*\beta_{rs}\gamma_{pq}^*\alpha_{ij}\delta_{kl}^*
+\alpha_{ij}\delta_{kl}^*\beta_{rs}\gamma_{pq}^*\alpha_{ij}\delta_{kl}^*\beta_{rs}\gamma_{pq}^*\\
&\hspace{1cm}+ \delta_{kl}\alpha_{ij}^*\gamma_{pq}\beta_{rs}^*\delta_{kl}\alpha_{ij}^*\gamma_{pq}\beta_{rs}^*
+\alpha_{ij}\delta_{kl}^*\beta_{rs}\gamma_{pq}^*\delta_{kl}\alpha_{ij}^*\gamma_{pq}\beta_{rs}^*.\numberthis{eq5}
\end{align*}
If (i) $i =p$ and $k \neq r$ or (ii) $i \neq r, k\neq p$ and $k \neq r$, then the Equation~\eqref{eq5} reduces to
\[
I +\delta_{kl}\alpha_{ij}^*\gamma_{pq}\beta_{rs}^* - \beta_{rs}\gamma_{pq}^*\alpha_{ij}\delta_{kl}^*
= \left[ E_{\delta_{kl}\alpha_{ij}^*}, E_{\beta_{rs}\gamma_{pq}^*}^*\right]^{-1}.
\]
If (i) $k=r$ and $i \neq p$ or (ii) $i \neq r$, $k \neq p$ and $i \neq p$, then the Equation~\eqref{eq5} reduces to
\[
I + \alpha_{ij}\delta_{kl}^*\beta_{rs}\gamma_{pq}^* - \gamma_{pq}\beta_{rs}^*\delta_{kl}\alpha_{ij}^*
= \left[ E_{\alpha_{ij}\delta_{kl}^*}, E_{\gamma_{pq}\beta_{rs}^*}^*\right]^{-1}.
\]
If  $i = r$ or $k = p$ or if $i \neq p$ and $k \neq r$ , then the Equation~\eqref{eq4} becomes
\begin{align*}
 \left[ [E_{\alpha_{ij}}, E_{\delta_{kl}}],\right.&\left. [ E_{\beta_{rs}}^*,E_{\gamma_{pq}}^*] \right](z,x,f)\\
& =I-\delta_{kl}\alpha_{ij}^*\beta_{rs}\gamma_{pq}^* -\alpha_{ij}\delta_{kl}^*\gamma_{pq}\beta_{rs}^*
+\gamma_{pq}\beta_{rs}^*\alpha_{ij}\delta_{kl}^*+\beta_{rs}\gamma_{pq}^*\delta_{kl}\alpha_{ij}^* \\
&\hspace{1cm} 
- \gamma_{pq}\beta_{rs}^*\alpha_{ij}\delta_{kl}^*\gamma_{pq}\beta_{rs}^* +\beta_{rs}\gamma_{pq}^*\delta_{kl}\alpha_{ij}^*\beta_{rs}\gamma_{pq}^*
+ \delta_{kl}\alpha_{ij}^*\beta_{rs}\gamma_{pq}^*\delta_{kl}\alpha_{ij}^*\\
&\hspace{1cm} - \alpha_{ij}\delta_{kl}^*\gamma_{pq}\beta_{rs}^* \alpha_{ij}\delta_{kl}^*
+\delta_{kl}\alpha_{ij}^*\beta_{rs}\gamma_{pq}^*\alpha_{ij}\delta_{kl}^*\gamma_{pq}\beta_{rs}^*\\
&\hspace{1cm} + \delta_{kl}\alpha_{ij}^*\beta_{rs}\gamma_{pq}^*\delta_{kl}\alpha_{ij}^*\beta_{rs}\gamma_{pq}^* + \alpha_{ij}\delta_{kl}^*\gamma_{pq}\beta_{rs}^*\alpha_{ij}\delta_{kl}^*\gamma_{pq}\beta_{rs}^* . \numberthis{eq6}
\end{align*}

If (i) $i = r$ and $k \neq p$ or (ii) $i \neq p, k\neq r$ and $k \neq p$, then the Equation~\eqref{eq6} reduces to
\[
I -\delta_{kl}\alpha_{ij}^*\beta_{rs}\gamma_{pq}^*  +\gamma_{pq}\beta_{rs}^*\alpha_{ij}\delta_{kl}^*
= \left[ E_{\delta_{kl}\alpha_{ij}^*}, E_{\gamma_{pq}\beta_{rs}^*}^*\right].
\]
If (i) $k=p$ and $i \neq r$ or (ii) $i \neq p, k \neq r$ and $i \neq r$, then the Equation~\eqref{eq6} reduces to
\[
I  -\alpha_{ij}\delta_{kl}^*\gamma_{pq}\beta_{rs}^* +\beta_{rs}\gamma_{pq}^*\delta_{kl}\alpha_{ij}^*
= \left[ E_{\alpha_{ij}\delta_{kl}^*}, E_{\beta_{rs}\gamma_{pq}^*}^*\right].\qedhere
\]
 \end{proof}

\vspace{2cm}

\noindent
{\it Acknowledgements.} The author is indebted to her advisor B. Sury for his valuable suggestions and careful verification of these computations at various stages during the preparation of this manuscript. She also likes to thank Ravi A. Rao for many illuminating discussions during the course of this work. The author would also like to thank Mohamed Barakat for introducing her to the power of GAP.


\end{document}